\documentclass{article}
\title{\textbf{Using a Grassmann graph to recover\\the underlying projective geometry}}
\usepackage{authblk}
\author[]{Ian Seong}
\date{}
\usepackage {parskip}
\usepackage [margin=2.5cm] {geometry}
\usepackage{enumitem}
\usepackage[utf8]{inputenc}
\usepackage[T1]{fontenc}
\usepackage{geometry}
\usepackage{amsmath, amssymb, amsthm}
\usepackage{hyperref}
\usepackage{graphicx}
\usepackage{mathtools}
\usepackage{relsize}
\usepackage{bm}
\usepackage{esvect}
\usepackage{tikz}
\usepackage{hhline}
\usepackage{enumitem}
\usepackage{comment}
\usepackage{ytableau}
\usepackage{pdflscape}
\usepackage{mathabx}

\def\FF{\mathbb{F}}
\def\Stab{{\rm{Stab}}}
\def\Fix{{\rm{Fix}}}
\newtheoremstyle{dotless}{}{}{\itshape}{}{\bfseries}{}{ }{}
\theoremstyle{dotless}
\newtheorem{theorem}{Theorem}[section]
\newtheorem{corollary}[theorem]{Corollary}
\newtheorem{lemma}[theorem]{Lemma}

\newtheoremstyle{dotlessdef}{}{}{}{}{\bfseries}{}{ }{}
\theoremstyle{dotlessdef}
\newtheorem{definition}[theorem]{Definition}
\newtheorem{problem}[theorem]{Problem}

\begin{document}

\maketitle
\begin{abstract}
    Let $n,k$ denote integers with $n>2k\geq 6$. Let $\mathbb{F}_q$ denote a finite field with $q$ elements, and let $V$ denote a vector space over $\mathbb{F}_q$ that has dimension $n$. The projective geometry $P_q(n)$ is the partially ordered set consisting of the subspaces of $V$; the partial order is given by inclusion. For the Grassmann graph $J_q(n,k)$ the vertex set consists of the $k$-dimensional subspaces of $V$. Two vertices of $J_q(n,k)$ are adjacent whenever their intersection has dimension $k-1$. The graph $J_q(n,k)$ is known to be distance-regular. Let $\partial$ denote the path-length distance function of $J_q(n,k)$. Pick two vertices $x,y$ in $J_q(n,k)$ such that $1<\partial(x,y)<k$. The set $P_q(n)$ contains the elements $x,y,x\cap y,x+y$. In our main result, we describe $x\cap y$ and $x+y$ using only the graph structure of $J_q(n,k)$. To achieve this result, we make heavy use of the Euclidean representation of $J_q(n,k)$ that corresponds to the second largest eigenvalue of the adjacency matrix.
\end{abstract}

\section{Introduction}
\label{intro}
In the topic of algebraic graph theory, there is a family of graphs said to be distance-regular \cite{BCN}. These graphs have certain parameters called the intersection numbers \cite[p.~126]{BCN}. The following question has been of interest for many years: is a given distance-regular graph uniquely determined up to isomorphism by its intersection numbers? This question has been resolved for several families of distance-regular graphs. The hypercubes \cite{egawa}, the half-cubes \cite[p.~195]{BCN}, the Hermitian forms graphs \cite{kite, ivanov}, and the odd graphs \cite{Moon} are all uniquely determined up to isomorphism by their intersection numbers. The Hamming graphs \cite{egawa}, the Johnson graphs \cite{chang, johnson}, and the bilinear forms graphs \cite{huang, cuypers, bimetsch, GK2} are uniquely determined up to isomorphism by their intersection numbers with a few exceptional cases. There are many other families of distance-regular graphs for which the question still remains open. One such family consists of the Grassmann graphs.

Here are some partial results on the uniqueness problem for the Grassmann graphs. Let $q$ denote a prime power and let $n,k$ denote integers with $n\geq 2k\geq 6$. In \cite{Metsch} Metsch showed that the Grassmann graph $J_q(n,k)$ is uniquely determined up to isomorphism by its intersection numbers with the following unresolved exceptions:
\begin{itemize}    
    \item $n=2k$ for $q\geq 2$;

    \item $n=2k+1$ for $q\geq 2$;

    \item $n=2k+2$ for $q\in \{2,3\}$;

    \item $n=2k+3$ for $q=2$.
\end{itemize}
In \cite{vDK} van Dam and Koolen constructed a family of distance-regular graphs, called the twisted Grassmann graphs $\overline{J}_q(2k+1,k)$; the graph $\overline{J}_q(2k+1,k)$ has the same intersection numbers as the Grassmann graph $J_q(2k+1,k)$, but the two graphs are not isomorphic. In \cite{GK} Gavrilyuk and Koolen showed that the Grassmann graph $J_q(2k,k)$ is uniquely determined up to isomorphism by its intersection numbers, provided that $k$ is sufficiently large. The lower bound on $k$ is given in \cite[p.~2]{GK}. 

This paper is about the relationship between the Grassmann graph $J_q(n,k)$ and its underlying projective geometry $P_q(n)$, and how to potentially recover $P_q(n)$ from the intersection numbers of $J_q(n,k)$. The geometry $P_q(n)$ is defined as follows. Let $\mathbb{F}_q$ denote a finite field with $q$ elements, and let $V$ denote a vector space over $\mathbb{F}_q$ that has dimension $n$. Let the set $P_q(n)$ consist of all the subspaces of $V$. Turn $P_q(n)$ into a partially ordered set with partial order given by inclusion.  The graph $J_q(n,k)$ is defined as follows. The vertex set $X$ of $J_q(n,k)$ consists of the $k$-dimensional subspaces of $V$. Two vertices of $J_q(n,k)$ are adjacent whenever their intersection has dimension $k-1$. Pick distinct $x,y\in X$. The geometry $P_q(n)$ contains the elements $x,y,x\cap y,x+y$. The main goal of the paper is to describe $x\cap y$ and $x+y$ using only the graph structure of $J_q(n,k)$.

We now summarize the main results of the paper. For the rest of this section, assume that $n>2k\geq 6$. We will use the notation $[m]=(q^m-1)/(q-1)$ for an integer $m$. Let $P_1$ denote the subset of $P_q(n)$ consisting of the elements that have dimension $1$. The cardinality of $P_1$ is $[n]$. We introduce a Euclidean space $E$ that has dimension $[n]-1$ and inner product $\left<\;,\;\right>$. We represent the elements of $P_1$ as vectors in $E$ as follows. For $s\in P_1$, we define the vector $\widehat{s}\in E$ such that the following (i)--(iv) are satisfied:
\begin{enumerate}[label=(\roman*)]
    \item $E=\text{Span}\bigl\{\widehat{s}\mid s\in P_1\bigr\}$;
    
    \item for $s\in P_1$, $\bigl\Vert\widehat{s}\bigr\Vert^2=[n]-1$;
    
    \item for distinct $s,t\in P_1$, $\Bigl<\widehat{s},\widehat{t}\Bigr>=-1$;
    
    \item $\mathlarger{\sum_{s\in P_1}\widehat{s}=0}$.
\end{enumerate}
Next, we represent the elements of $P_q(n)$ as vectors in $E$ as follows. For $u\in P_q(n)$ we define $\widehat{u}\in E$ as 
\begin{equation*}
    \widehat{u}=\sum_{\substack{s\in P_1\\ s\subseteq u}}\widehat{s}.
\end{equation*}

We show that for $u,v\in P_q(n)$, 
\begin{equation*}
    \bigl<\widehat{u},\widehat{v}\bigr> = [n][h]-[i][j],
\end{equation*}
where $i=\dim u$, $j=\dim v$, $h=\dim (u\cap v)$. Let $\partial$ denote the path-length distance function of $J_q(n,k)$. We show that for $x,y\in X$,
\begin{equation*}
    \bigl<\widehat{x},\widehat{y}\bigr>=[n][k-i]-[k]^2,
\end{equation*}
where $i=\partial(x,y)$. For $x\in X$, let $\Gamma(x)$ denote the set of vertices in $J_q(n,k)$ that are adjacent to $x$. We show that   
\begin{equation*}
    \sum_{z\in \Gamma(x)}\widehat{z}=\theta_1\widehat{x}
\end{equation*}
where $\theta_1=q^2[k-1][n-k-1]-1$. The scalar $\theta_1$ is the second largest eigenvalue of the adjacency matrix of $J_q(n,k)$. We show that the vectors $\bigl\{\widehat{x}\mid x\in X\bigr\}$ span $E$. Using the above facts, we show that the vectors $\bigl\{\widehat{x}\mid x\in X\bigr\}$ give a Euclidean representation of $J_q(n,k)$; see Definition \ref{Euclidean} below. Pick $x,y\in X$ such that $1<\partial(x,y)<k$. Using $x$ and $y$ we construct two vectors $B_{xy}, C_{xy}\in E$ as follows. Define the sets
\begin{equation*}
    \begin{aligned}
        \mathcal{B}_{xy}&=\{z\in \Gamma(x)\mid \partial(y,z)=\partial(x,y)+1\},\\
    \mathcal{C}_{xy}&=\{z\in \Gamma(x)\mid \partial(y,z)=\partial(x,y)-1\}.
    \end{aligned}
\end{equation*}
Define the vectors
\begin{equation*}
    B_{xy}=\sum_{z\in \mathcal{B}_{xy}}\widehat{z},\qquad \qquad C_{xy}=\sum_{z\in \mathcal{C}_{xy}}\widehat{z}.
\end{equation*}
We show that the vectors
\begin{equation}
\label{intro1}
    \widehat{x},\qquad \widehat{y},\qquad \widehat{x\cap y},\qquad \widehat{x+y}
\end{equation}
and the vectors  
\begin{equation}
\label{intro2}
    \widehat{x},\qquad \widehat{y},\qquad B_{xy},\qquad C_{xy}
\end{equation}
are both bases for the same subspace of $E$. We find all the inner products between
\begin{enumerate}[label=(\roman*)]
    \item pairs of vectors in (\ref{intro1});

    \item pairs of vectors in (\ref{intro2});

    \item a vector in (\ref{intro1}) and a vector in (\ref{intro2}).
\end{enumerate}
Using these inner products, we obtain the transition matrices between the bases (\ref{intro1}) and (\ref{intro2}). This yields
    \begin{equation*}
        \begin{aligned}
            \widehat{x\cap y}&=\frac{[k-i][n-k-1]}{q^{k-1}[n-2k]}\widehat{x}+\frac{[k-i]}{q^{k-i+1}[i-1][n-2k]}\widehat{y}+\frac{-1}{q^{k+i}[n-2k]}B_{xy}+\frac{-[k-i]}{q^k[i-1][n-2k]}C_{xy},\\
            \widehat{x+y}&=\frac{-[k-1][n-k-i]}{q^{k-i-1}[n-2k]}\widehat{x}+\frac{-[n-k-i]}{q^{k-2i+1}[i-1][n-2k]}\widehat{y}+\frac{1}{q^k[n-2k]}B_{xy}+\frac{[n-k-i]}{q^{k-i}[i-1][n-2k]}C_{xy},
        \end{aligned}
    \end{equation*}
    where $i=\partial(x,y)$. The above equations are our main result. Near the end of the paper, we obtain some variations on the main result using some relatives of $B_{xy}, C_{xy}$. We also have some more comments about the uniqueness problem for $J_q(n,k)$.

    The paper is organized as follows. In Sections 2--3, we present some basic facts about the Grassmann graph $J_q(n,k)$ and the projective geometry $P_q(n)$. In Sections 4--6, we represent the elements of $P_q(n)$ as vectors in the Euclidean space $E$. In Sections 7--8, we discuss the subspace of $E$ spanned by the basis (\ref{intro1}). In Section 9, we construct the vectors $B_{xy}$ and $C_{xy}$. In Section 10, we compute the inner products. In Section 11, we prove that the vectors in (\ref{intro2}) form a basis for the subspace spanned by the basis (\ref{intro1}). We give the transition matrices between the bases (\ref{intro1}) and (\ref{intro2}). In Sections 12--13, we discuss some relatives of $B_{xy}$ and $C_{xy}$. In Section 14, we give some more comments about the uniqueness problem for $J_q(n,k)$. Section 15 is an appendix that contains some linear algebra facts that are used throughout the paper.

\section{The Grassmann graph $\Gamma$}
\label{prelim}
Let $\Gamma=(X,\mathcal{E})$ denote a finite undirected graph that is connected, without loops or multiple edges, with vertex set $X$, edge set $\mathcal{E}$, and path-length distance function $\partial$. Two vertices are said to be adjacent whenever they form an edge. The diameter $d$ of $\Gamma$ is defined as $d=\max\{\partial(x,y)\mid x,y\in X\}$. For $x\in X$, define the set $\Gamma(x)=\{y\in X\mid \partial(x,y)=1\}$. 
We say that $\Gamma$ is \textit{regular with valency $\kappa$} whenever $|\Gamma(x)|=\kappa$ for all $x\in X$. 
We say that $\Gamma$ is \textit{distance-regular} whenever for all non-negative integers $h,i,j$ and all $x,y\in X$ such that $\partial(x,y)=h$, the cardinality of the set $\{z\in X\mid\partial(x,z)=i, \partial(y,z)=j\}$ depends only on $h,i,j$. This cardinality is denoted by $p_{i,j}^{h}$. If $\Gamma$ is distance-regular, then $\Gamma$ is regular with valency $\kappa=p^{0}_{1,1}$. For the rest of this section, we assume that $\Gamma$ is distance-regular with diameter $d\geq 3$.

Define 
\begin{equation*}
        b_i=p^{i}_{1,i+1}\; \;  (0\leq i<d),\qquad \qquad a_i=p^{i}_{1,i} \; \; (0\leq i\leq d),\qquad \qquad c_i=p^{i}_{1,i-1}\; \; (0<i\leq d).
\end{equation*}
Note that $b_0=\kappa$, $a_0=0$, $c_1=1$. Also note that  
\begin{equation}
\label{valency}
    b_i+a_i+c_i=\kappa\qquad (0\leq i\leq d),
\end{equation}
where $c_0=0$ and $b_d=0$.

We call $b_i$, $a_i$, $c_i$ the {\it intersection numbers of $\Gamma$}.

By the {\it eigenvalues of $\Gamma$} we mean the
eigenvalues of the adjacency matrix of $\Gamma$. Since $\Gamma$ is distance-regular, by \cite[p.~128]{BCN}, $\Gamma$ has $d+1$ eigenvalues; we denote these eigenvalues by
\begin{equation}
\nonumber
    \theta_0>\theta_1>\cdots>\theta_d.
\end{equation} By \cite[p.~129]{BCN}, $\theta_0=\kappa$.

This paper is about a class of distance-regular graphs
called the Grassmann graphs. These graphs are
defined as follows. Let $\mathbb{F}_q$ denote a finite field with $q$ elements, and let $n,k$ denote positive integers such that $n> k$. Let $V$ denote an $n$-dimensional vector space over $\FF_{q}$. The Grassmann graph $J_q(n,k)$ has vertex set $X$ consisting of the $k$-dimensional subspaces of $V$. Vertices $x,y$ of $J_q(n,k)$ are adjacent whenever $x \cap y$ has dimension $k-1$.

According to \cite[p.~268]{BCN}, the graphs $J_q(n,k)$ and $J_q(n,n-k)$ are isomorphic. Without loss of generality, we may assume $n\geq 2k$. Under this assumption, the diameter of $J_q(n,k)$ is equal to $k$. (See \cite[Theorem~9.3.3]{BCN}.) The case $n=2k$ is somewhat special, so throughout this paper we assume that $n>2k$. 

For the rest of this paper, we assume that $\Gamma$ is the Grassmann graph $J_q(n,k)$ with $k\geq 3$.

In what follows, we will use the notation 
\begin{equation}
\label{notation}
    [m]=\frac{q^{m}-1}{q-1} \qquad \qquad (m\in \mathbb{Z}).
\end{equation} 

By \cite[Theorem~9.3.2]{BCN}, the valency of $\Gamma$ is 
\begin{equation}
\label{kappa}
  \kappa=q[k][n-k].  
\end{equation}

By \cite[Theorem~9.3.3]{BCN}, the intersection numbers of $\Gamma$ are
\begin{equation}
    \label{sizebc}
    b_i=q^{2i+1}[k-i][n-k-i],\qquad \qquad c_i=[i]^2\qquad \qquad (0\leq i\leq k).
\end{equation}

By {\rm{\cite[Theorem~9.3.3]{BCN}}}, the eigenvalues of $\Gamma$ are
 \begin{equation}
 \label{eigenvalues}
     \theta_i=q^{i+1}[k-i][n-k-i]-[i] \qquad \qquad (0\leq i\leq k).
 \end{equation}

 The given ordering of the eigenvalues is known to be $Q$-polynomial in the sense of \cite[p.~135]{BCN}.

\section{The projective geometry $P_q(n)$}
To study the graph $\Gamma$, it is helpful to view its vertex set $X$ as a subset of a certain poset $P_q(n)$, which is defined as follows. 
\begin{definition}
    Let the poset $P_q(n)$ consist of the subspaces of $V$, together with the partial order given by inclusion. This poset $P_q(n)$ is called the \emph{projective geometry}.
\end{definition}
For the rest of the paper, we abbreviate $P=P_q(n)$. In this section we present some lemmas about the poset $P$. 

\begin{lemma}{\rm{\cite[p.~47]{Axler}}}
    \label{modularity}
    For $u,v\in P$ we have 
    \begin{equation*}
        \dim u+\dim v=\dim (u\cap v)+\dim(u+v).
    \end{equation*}
\end{lemma}

\begin{lemma}
    \label{linin}
    Let $u,v\in P$. Let the subset $\mathcal{R}\subseteq V$ form a basis for $u\cap v$. Extend the basis $\mathcal{R}$ to a basis $\mathcal{R}\cup \mathcal{S}$ for $u$, and extend the basis $\mathcal{R}$ to a basis $\mathcal{R}\cup \mathcal{T}$ for $v$. Then $\mathcal{R}\cup \mathcal{S}\cup \mathcal{T}$ forms a basis for the subspace $u+v$.
\end{lemma}
\begin{proof}
    Since the set $\mathcal{R}\cup \mathcal{S}$ is a basis for $u$ and the set $\mathcal{R}\cup \mathcal{T}$ is a basis for $v$, the set $\mathcal{R}\cup \mathcal{S}\cup \mathcal{T}$ spans $u+v$.

    By Lemma \ref{modularity}, 
    \begin{equation*}
        \dim (u+v)=\dim u+\dim v-\dim (u\cap v)=\vert \mathcal{R}\vert + \vert \mathcal{S}\vert + \vert \mathcal{T}\vert.
    \end{equation*} 
    The result follows.
\end{proof}

For $0\leq \ell\leq n$, let the set $P_{\ell}$ consist of the $\ell$-dimensional subspaces of $V$. 
Note that $X=P_k$. Also note that $P_0=\{0\}$ and $P_n=\{V\}$.

\begin{lemma}{\rm{\cite[p.~269]{BCN}}}
\label{introlem0}
For $x,y\in X$ the dimension of $x\cap y$ is $k-\partial(x,y)$.
\end{lemma}

\begin{lemma}
\label{sumdim}
For $x,y\in X$ the dimension of $x+y$ is $k+\partial(x,y)$.
\end{lemma}
\begin{proof}
    Routine using Lemmas \ref{modularity}, \ref{introlem0}.
\end{proof}

\begin{lemma}
\label{contain}
Let $x,y,z\in X$ satisfy $\partial(x,y)=\partial(x,z)+\partial(z,y)$. Then $x\cap y\subseteq z\subseteq x+y$.
\end{lemma}
\begin{proof}
    Routine from linear algebra.
\end{proof}

\begin{definition}
\label{defs}
For $u\in P$ define the set 
\begin{equation*}
  \Omega(u)=\{s\in P_1\mid s\subseteq u\}.  
\end{equation*}
\end{definition}

Note that $\Omega(V)=P_1$.

For $u,v\in P$ we have 
\begin{equation}
\label{scap}
    \Omega(u)\cap \Omega(v)=\Omega(u\cap v).
\end{equation}

The following notation will be useful. For an integer $m\geq 0$, define
\begin{equation*}
    [m]^{!}=[m][m-1]\cdots[2][1].
\end{equation*}

We interpret $[0]^{!}=1$.

For integers $0\leq r\leq m$, define
\begin{equation*}
    \begin{bmatrix}
    m\\r
    \end{bmatrix}=\frac{[m]^{!}}{[r]^{!}[m-r]^{!}}.
\end{equation*}

\begin{lemma}{\rm{\cite[Theorem~9.3.2]{BCN}}}
    \label{subspacesize}
    For integers $0\leq r\leq m$, $\begin{bmatrix}
    m\\r
\end{bmatrix}$ is equal to the number of $r$-dimensional subspaces of a given element in $P_m$.
\end{lemma}

By Lemma \ref{subspacesize}, we find that for all $u\in P$,
\begin{equation}
    \label{orbitsize}
    \bigl|\Omega(u)\bigr|=[m],
\end{equation}
where $u\in P_{m}$.

Letting $u=V$ in (\ref{orbitsize}), we get
\begin{equation*}
    \vert P_1\vert = \bigl\vert \Omega(V)\bigr\vert =[n].
\end{equation*}

We now comment on the symmetries of $P$. Recall that the general linear group $GL(V)$ consists of the invertible $\mathbb{F}_q$-linear maps from $V$ to $V$. The action of $GL(V)$ on $V$ induces a permutation action of $GL(V)$ on the set $P$. This permutation action respects the partial order on $P$. The orbits of the action are $P_\ell$ for $0\leq\ell\leq n$.

\begin{lemma}
\label{presdist}
    For $x,y\in X$ and $\sigma\in GL(V)$, we have $\partial(x,y)=\partial(\sigma(x),\sigma(y))$.
\end{lemma}

\begin{proof}
    Immediate from Lemma \ref{introlem0} and the fact that the $\sigma$-action preserves dimension.
\end{proof}

\section{Representing $P$ using a Euclidean space $E$}
\label{Euc1}

In this section we represent the elements of $P$ as vectors in a Euclidean space. We will do this in two stages. In the first stage we consider the elements of $P_1$. 

Let $E$ denote a Euclidean space with dimension $[n]-1$ and bilinear form $\left<\; ,\;\right>$. Recall the notation $\Vert \nu\Vert^2=\left<\nu,\nu\right>$ for any $\nu\in E$. We define a function \begin{equation}
    \begin{aligned}
    \label{basichat}
        P_1&\rightarrow E\\s&\mapsto\widehat{s}
    \end{aligned}
\end{equation}
that satisfies the following conditions (C1) $-$ (C4):

\begin{enumerate}[label=(C\arabic*)]
    \item $E=\text{Span}\bigl\{\widehat{s}\mid s\in P_1\bigr\}$;
    
    \item for $s\in P_1$, $\bigl\Vert\widehat{s}\bigr\Vert^2=[n]-1$;
    
    \item for distinct $s,t\in P_1$, $\Bigl<\widehat{s},\widehat{t}\Bigr>=-1$;
    
    \item $\mathlarger{\sum_{s\in P_1}\widehat{s}=0}$.
\end{enumerate}

Next, we extend the function (\ref{basichat}) to a function
\begin{equation}
    \begin{aligned}
    \label{hardhat}
        P&\rightarrow E \\
        u&\mapsto \widehat{u}
    \end{aligned}
\end{equation}
such that for all $u\in P$, 
\begin{equation}
\label{hardhatsum}
    \widehat{u}=\sum_{s\in \Omega(u)}\widehat{s}.
\end{equation}

Note that $\widehat{u}=0$ if $u\in P_0$ or $u\in P_n$. 

In (C4), we gave a linear dependence on $\bigl\{\widehat{s}\bigr\}_{s\in P_1}$. Next we show that (C4) is essentially the only linear dependence among $\bigl\{\widehat{s}\bigr\}_{s\in P_1}$.

We use the following notation. For sets $\mathcal{R}\subseteq \mathcal{S}$, define $\mathcal{S}\setminus \mathcal{R}$ to be the complement of $\mathcal{R}$ in $\mathcal{S}$.

\begin{lemma}
\label{sum0converse}
Given real numbers $\{\alpha_s\}_{s\in P_1}$ the following are equivalent:
\begin{enumerate}[label=\rm{(\roman*)}]
    \item $0=\mathlarger{\sum_{s\in P_1}\alpha_s\widehat{s}}$;

    \item $\alpha_s$ is independent of $s$ for $s\in P_1$.
\end{enumerate}
\end{lemma}
\begin{proof}
(i)$\Rightarrow$(ii)
    Pick $t\in P_1$. Referring to (C4), multiply each side by $\alpha_t$ to obtain
    \begin{equation}
    \label{wv}
        0=\sum_{s\in P_1}\alpha_{t}\widehat{s}.
    \end{equation}
    
    Subtract (\ref{wv}) from the equation in (i) to obtain 
    \begin{equation*}
        0=\sum_{s\in P_1\setminus \{t\}}(\alpha_s-\alpha_t)\widehat{s}.
    \end{equation*}

    By Lemma \ref{sum0basis} in the Appendix, the vectors $\bigl\{\widehat{s}\bigr\}_{s\in P_1\setminus \{t\}}$ form a basis for $E$, so they are linearly independent. Hence $\alpha_s-\alpha_{t}=0$ for all $s\in P_1$. The result follows.

    (ii)$\Rightarrow$(i) Immediate from (C4).
\end{proof}

\begin{lemma}
\label{sigmau}
The Euclidean space $E$ becomes a $GL(V)$-module such that for all $u\in P$ and $\sigma\in GL(V)$, 
\begin{equation*}
    \sigma\bigl(\widehat{u}\bigr)=\widehat{\sigma(u)}.
\end{equation*}
\end{lemma}
\begin{proof}
It suffices to show that there exists a group homomorphism $\Theta:GL(V)\rightarrow GL(E)$ such that for all $u\in P$ and $\sigma\in GL(V)$, 
\begin{equation}
\label{thetau}
    \Theta(\sigma)\bigl(\widehat{u}\bigr)=\widehat{\sigma(u)}.
\end{equation}

Let $\sigma\in GL(V)$. We first show that there exists a unique $\Theta(\sigma)\in GL(E)$ such that (\ref{thetau}) holds for all $u\in P_1$.

By construction, the dimension of $E$ is $|P_1|-1$. 

Note that $\sigma$ acts as a permutation on $P_1$. By this along with (C1) and (C4), we obtain
\begin{equation*}
    E=\text{Span}\biggl\{\widehat{\sigma(s)}\;\Bigl\vert\; s\in P_1\biggr\}, \qquad \qquad \quad \sum_{s\in P_1}\widehat{\sigma(s)}=\sum_{s\in P_1}\widehat{s}=0.
\end{equation*}

By these comments and Lemma \ref{sum0} in the Appendix, there exists a unique $\Theta(\sigma)\in GL(E)$ that satisfies (\ref{thetau}) for all $u\in P_1$.

Next we show that $\Theta(\sigma)$ satisfies (\ref{thetau}) for all $u\in P$. By (\ref{hardhatsum}),
\begin{equation*}
        \Theta(\sigma)\bigl(\widehat{u}\bigr)=\Theta(\sigma)\left(\sum_{s\in \Omega(u)}\widehat{s}\right)
        =\sum_{s\in \Omega(u)}\Theta(\sigma)\bigl(\widehat{s}\bigr)
        =\sum_{s\in \Omega(u)}\widehat{\sigma(s)}
        =\widehat{\sigma(u)}.
\end{equation*}

Next we show that $\Theta$ is a group homomorphism. For $\sigma,\tau\in GL(V)$, we show that $\Theta(\sigma)\Theta(\tau)=\Theta(\sigma\tau)$. 

For $u\in P$ we have 
\begin{equation*}
    \Theta(\sigma)\Theta(\tau)\bigl(\widehat{u}\bigr)=\Theta(\sigma)\Bigl(\widehat{\tau(u)}\Bigr)=\widehat{\sigma\tau(u)}=\Theta(\sigma\tau)\bigl(\widehat{u}\bigr).
\end{equation*}

Therefore, $\Theta(\sigma)\Theta(\tau)=\Theta(\sigma\tau)$.

We have shown that there exists a unique group homomorphism $\Theta:GL(V)\rightarrow GL(E)$ that satisfies (\ref{thetau}) for all $\sigma\in GL(V)$ and $u\in P$. Consequently, $E$ becomes a $GL(V)$-module such that $\sigma\bigl(\widehat{u}\bigr)=\widehat{\sigma(u)}$
for all $u\in P$ and $\sigma\in GL(V)$.
\end{proof}

\begin{lemma}
For $\mu,\nu\in E$ and $\sigma\in GL(V)$,
\begin{equation*}
    \left<\mu,\nu\right>=\bigl<\sigma(\mu),\sigma(\nu)\bigr>.
\end{equation*}
\end{lemma}
\begin{proof}
In view of (C1), we may assume without loss that $\mu=\widehat{s}$ and $\nu=\widehat{t}$ where $s,t\in P_1$. The result is a routine consequence of (C2) and (C3).
\end{proof}

\begin{lemma}
\label{subgroup1}
    Pick a vector $\mu\in E$ and write 
    \begin{equation}
        \label{lincom}
            \mu=\sum_{s\in P_1}\alpha_s \widehat{s}.
    \end{equation}

    For $\sigma\in GL(V)$ that fixes $\mu$, $\alpha_s=\alpha_{\sigma(s)}$ for all $s\in P_1$.
\end{lemma}

\begin{proof}
    Referring to (\ref{lincom}), apply $\sigma^{-1}$ to each side and evaluate the result using Lemma \ref{sigmau} to get 
    \begin{equation}
    \label{hlincom}
        \mu=\sum_{s\in P_1}\alpha_s \widehat{\sigma^{-1}\left(s\right)}.
    \end{equation}
    We make a change of variables. In (\ref{hlincom}), replace $s$ by $\sigma(s)$ to get
    \begin{equation}
        \label{hlincomalt}
        \mu=\sum_{s\in P_1}\alpha_{\sigma(s)} \widehat{s}.
    \end{equation}
    Subtract (\ref{lincom}) from (\ref{hlincomalt}) to obtain 
    \begin{equation*}
        0=\sum_{s\in P_1}(\alpha_{\sigma(s)}-\alpha_{s})\widehat{s}.
    \end{equation*}

    By Lemma \ref{sum0converse}, the scalar $\alpha_{\sigma(s)}-\alpha_{s}$ is independent of $s\in P_1$. Denote this common value by $\beta$. Then
    \begin{equation}
    \label{betasum}
        \sum_{s\in P_1}\left(\alpha_{\sigma(s)}-\alpha_{s}\right)=\sum_{s\in P_1}\beta=\beta \vert P_1\vert.
    \end{equation}

    Note that 
    \begin{equation}
    \label{betasum0}
        \sum_{s\in P_1}\left(\alpha_{\sigma(s)}-\alpha_{s}\right)=\sum_{s\in P_1}\alpha_{s}-\sum_{s\in P_1}\alpha_{s}=0.
    \end{equation} 
    By (\ref{betasum}), (\ref{betasum0}) and the fact that $\vert P_1\vert$ is nonzero, we obtain $\beta=0$. The result follows.
\end{proof}

\begin{corollary}
\label{subgroup2}
    Pick a vector $\mu\in E$ and write 
    \begin{equation*}
            \mu=\sum_{s\in P_1}\alpha_s \widehat{s}.
    \end{equation*}

    For a subgroup $H$ of $GL(V)$ the following are equivalent:  
    \begin{enumerate}[label=\rm{(\roman*)}]
        \item every element of $H$ fixes $\mu$;

        \item for $s,t\in P_1$ that are contained in the same $H$-orbit, $\alpha_s=\alpha_t$.
    \end{enumerate}
\end{corollary}
\begin{proof}
Immediate from Lemma \ref{subgroup1}.    
\end{proof}

Our next general goal is to restate Corollary \ref{subgroup2} using a different point of view.

\begin{definition}
    Let $\mathcal{S}$ denote a subset of $P_1$. By the \emph{characteristic vector of $\mathcal{S}$} we mean the vector
    \begin{equation*}
        \sum_{s\in \mathcal{S}}\widehat{s}.
    \end{equation*}
    Note that this characteristic vector is contained in $E$.
\end{definition}

\begin{lemma}
\label{char0}
    For a subset $\mathcal{S}\subseteq P_1$ the sum of the following vectors is zero:
    \begin{enumerate}[label={\rm{(\roman*)}}]
        \item the characteristic vector of $\mathcal{S}$;

        \item the characteristic vector of the set $P_1\setminus \mathcal{S}$.
    \end{enumerate}
\end{lemma}

\begin{proof}
    Immediate from (C4).
\end{proof}

\begin{lemma}
\label{character}
    Pick a vector $\mu\in E$ and a subgroup $H$ of $GL(V)$. Then the following are equivalent:
    \begin{enumerate}[label={\rm{(\roman*)}}]
        \item every element of $H$ fixes $\mu$;

        \item $\mu$ is contained in the span of the characteristic vectors of the $H$-orbits in $P_1$.
    \end{enumerate}
\end{lemma}

\begin{proof}
    Immediate from Corollary \ref{subgroup2}.
\end{proof}

\section{The stabilizer of an element in $P$}

In this section we pick an element in $P$ and consider its stabilizer in $GL(V)$. We describe the orbits of the stabilizer acting on $P$. We also consider how the stabilizer acts on $E$. 

For $u\in P$, let $\Stab(u)$ denote the subgroup of $GL(V)$ consisting of the elements that fix $u$. We call $\Stab(u)$ the {\it stabilizer of $u$ in $GL(V)$}. Note that $\Stab(0)=GL(V)$ and $\Stab(V)=GL(V)$.

\begin{lemma}
\label{staborbit}
    For $u,v,v'\in P$ the following are equivalent:
    \begin{enumerate}[label=\rm{(\roman*)}]
        \item $\dim v=\dim v'$ and $\dim (u\cap v)=\dim (u\cap v')$;

        \item the subspaces $v$ and $v'$ are contained in the same orbit of the $\Stab(u)$-action on $P$.
    \end{enumerate}
\end{lemma}

\begin{proof}
    (i)$\Rightarrow$(ii) 
    We display an element $\sigma\in \Stab(u)$ that sends $v\mapsto v'$. 

    Let the subset $\mathcal{R}\subseteq V$ form a basis for $u\cap v$. 
    Extend the basis $\mathcal{R}$ to a basis $\mathcal{R}\cup \mathcal{S}$ for $u$.
    Extend the basis $\mathcal{R}$ to a basis $\mathcal{R}\cup \mathcal{T}$ for $v$.
    By Lemma \ref{linin}, $\mathcal{R}\cup \mathcal{S}\cup \mathcal{T}$ is a basis for $u+v$.
    Extend the basis $\mathcal{R}\cup \mathcal{S}\cup \mathcal{T}$ to a basis $\mathcal{R}\cup \mathcal{S}\cup \mathcal{T}\cup \mathcal{Q}$ for $V$.
    
    Let the subset $\mathcal{R}'\subseteq V$ form a basis for $u\cap v'$.
    Extend the basis $\mathcal{R}'$ to a basis $\mathcal{R}'\cup \mathcal{S}'$ for $u$. 
    Extend the basis $\mathcal{R}'$ to a basis $\mathcal{R}'\cup \mathcal{T}'$ for $v'$.
    By Lemma \ref{linin}, $\mathcal{R}'\cup \mathcal{S}'\cup \mathcal{T}'$ is a basis for $u+v'$. 
    Extend the basis $\mathcal{R}'\cup \mathcal{S}'\cup \mathcal{T}'$ to a basis $\mathcal{R}'\cup \mathcal{S}'\cup \mathcal{T}'\cup \mathcal{Q}'$ for $V$.

    By linear algebra, there exists $\sigma\in GL(V)$ that sends $\mathcal{R}\mapsto \mathcal{R}'$, $\mathcal{S}\mapsto \mathcal{S}'$, $\mathcal{T}\mapsto \mathcal{T}'$, $\mathcal{Q}\mapsto \mathcal{Q}'$. By construction, $\sigma$ is contained in $\Stab(u)$ and sends $v\mapsto v'$. 
      
    (ii)$\Rightarrow$(i) Let $\sigma\in \Stab(u)$ send $v\mapsto v'$. By linear algebra, $\sigma$ sends $u\cap v\mapsto u\cap v'$. The result follows since dimensions are left invariant under the $\sigma$-action.
\end{proof}

\begin{corollary}
    \label{staborbitP1}
    For $u\in P$ the following hold.
    \begin{enumerate}[label={\rm{(\roman*)}}]
        \item If $u\neq 0$ and $u\neq V$, then the $\Stab(u)$-action on $P_1$ has two orbits, $\Omega(u)$ and $P_1\setminus \Omega(u)$.

        \item If $u=0$ or $u=V$, then the $\Stab(u)$-action on $P_1$ has a single orbit.
    \end{enumerate}
\end{corollary}

\begin{proof}
    Use Lemma \ref{staborbit}.
\end{proof}

For $u\in P$, let $\Fix(u)$ denote the subspace of $E$ consisting of the vectors that are fixed by every element of $\Stab(u)$. Note that $\Fix(0)=0$ and $\Fix(V)=0$.

\begin{lemma}
\label{fixspan}
For $u\in P$ the subspace $\Fix(u)$ is spanned by $\widehat{u}$.
\end{lemma}
\begin{proof}
First assume that $u=0$ or $u=V$. Then the result holds since $\widehat{u}=0$. Next assume that $u\neq 0$ and $u\neq V$. By Corollaries \ref{subgroup2} and \ref{staborbitP1}, the subspace $\Fix(u)$ is spanned by the characteristic vectors of $\Omega(u)$ and $P_1\setminus \Omega(u)$. By Lemma \ref{char0} and Lemma \ref{sum0basis} in the Appendix, the characteristic vector of $\Omega(u)$ forms a basis for $\Fix(u)$. The result follows.
\end{proof}

\section{A Euclidean representation of $\Gamma$}
\label{euclideanrepresentation}

In this section we recall the notion of a Euclidean
representation of $\Gamma$. We then show that the Euclidean space $E$, together with the restriction of the map (\ref{hardhat}) to $X$ is a Euclidean representation of $\Gamma$.

\begin{definition}{\cite[Lecture~12]{TerCoursenote}}
\label{Euclidean}
    By a \emph{Euclidean representation of $\Gamma$}, we mean a pair $(F,\rho)$ such that $F$ is a nonzero Euclidean space, and $\rho:X\rightarrow F$ is a map that satisfies the following (i)--(iii):
    \begin{enumerate}[label={\rm{(\roman*)}}]
        \item $F$ is spanned by $\{\rho(x)\mid x\in X\}$;
        \item for all $x,y\in X$, the inner product $\bigl<\rho(x),\rho(y)\bigr>$ depends only on $\partial(x,y)$;
        \item there exists $\vartheta\in \mathbb{R}$ such that for all $x\in X$, 
        \begin{equation*}
            \sum_{z\in \Gamma(x)}\rho(z)=\vartheta\rho(x).
        \end{equation*}
    \end{enumerate}
    We have some comments about Definition $\ref{Euclidean}$. Let $(F,\rho)$ denote a Euclidean representation of $\Gamma$. It is shown in \cite[Lecture~13]{TerCoursenote} that the scalar $\vartheta$ in Definition $\ref{Euclidean}$(iii) is an eigenvalue of $\Gamma$. We call $\vartheta$ the \emph{associated eigenvalue}. For any eigenvalue $\theta$ of $\Gamma$, it is shown in \cite[Lecture~13]{TerCoursenote} how the corresponding eigenspace gives a Euclidean representation of $\Gamma$ that is associated with $\theta$.
\end{definition}

Recall the Euclidean space $E$ from Section \ref{Euc1}. Our next goal is to show that the Euclidean space $E$, together with the restriction of the map (\ref{hardhat}) to $X$ is a Euclidean representation of $\Gamma$.

\begin{lemma}
    \label{introlem1}
    For $u,v\in P$ we have 
    \begin{equation*}
        \bigl<\widehat{u},\widehat{v}\bigr>=[n][h]-[i][j],
    \end{equation*}
    where $u\in P_i$, $v\in P_j$, and $u\cap v\in P_h$.
\end{lemma}
    
\begin{proof}
By (\ref{orbitsize}),
\begin{equation}
\label{suv}
    \bigl\vert\Omega(u)\bigr\vert=[i],\qquad \qquad \bigl\vert\Omega(v)\bigr\vert=[j].
\end{equation}

    In view of (\ref{hardhatsum}), write 
    \begin{equation}
    \label{sumsplit}
        \bigl<\widehat{u},\widehat{v}\bigr>=\left<\sum_{s\in \Omega(u)}\widehat{s},\sum_{t\in \Omega(v)}\widehat{t}\right>=\sum_{s\in \Omega(u)\cap \Omega(v)}\bigl\Vert\widehat{s}\bigr\Vert^2+\sum_{\substack{s\in \Omega(u), t\in \Omega(v)\\ s\neq t}}\Bigl<\widehat{s},\widehat{t}\Bigr>.
    \end{equation}
    
    By (\ref{scap}) and (\ref{orbitsize}), 
    \begin{equation}
    \label{scombine}
        \bigl\vert \Omega(u) \cap \Omega(v)\bigr\vert = \bigl\vert \Omega(u \cap v)\bigr\vert = [h].
    \end{equation}
    
    By (\ref{suv}) and (\ref{scombine}),
    \begin{equation}
    \label{notcard}
        \Bigl\vert\bigl\{(s,t)\in \Omega(u)\times \Omega(v)\;\bigl\vert\; s\neq t\bigr\}\Bigr\vert=\bigl\vert \Omega(u)\bigr\vert \bigl\vert \Omega(v)\bigr\vert -\bigl\vert \Omega(u\cap v)\bigr\vert=[i][j]-[h].
    \end{equation}
    
    The result follows from (\ref{sumsplit})--(\ref{notcard}), along with (C2), (C3).
\end{proof}

\begin{corollary}
    \label{introcor}
    For $u\in P$ we have 
    \begin{equation*}
        \bigl\Vert\widehat{u}\bigr\Vert^2=q^{i}[i][n-i],
    \end{equation*}
    where $u\in P_i$.
\end{corollary}

\begin{proof}
    Set $u=v$ in Lemma \ref{introlem1}.
\end{proof}

\begin{lemma}
    \label{cor1}
    For $x,y\in X$ we have 
    \begin{equation*}
        \bigl<\widehat{x},\widehat{y}\bigr>=[n][k-i]-[k]^2,
    \end{equation*}
    where $i=\partial(x,y)$.
\end{lemma}
\begin{proof}
    Immediate from Lemmas \ref{introlem0}, \ref{introlem1}.
\end{proof}

\begin{corollary}
    \label{cor2}
    For $x\in X$, 
    \begin{equation*}
        \bigl\Vert\widehat{x}\bigr\Vert^2=q^k[k][n-k].
    \end{equation*}
\end{corollary}
\begin{proof}
    Set $i=0$ in Lemma \ref{cor1}.
\end{proof}

\begin{lemma}
\label{introlem2}
For $x\in X$, $$\sum_{z\in \Gamma(x)}\widehat{z}=\theta_1\widehat{x},$$ where $\theta_1$ is from {\rm(\ref{eigenvalues})}.   
\end{lemma}

\begin{proof}
    Note that the set $\Gamma(x)$ is invariant under $\Stab(x)$. Hence, the vector 
\begin{equation}
\label{gammasum}
    \sum_{z\in \Gamma(x)}\widehat{z}
\end{equation} is fixed by every element of $\Stab(x)$. Therefore, the vector (\ref{gammasum}) is contained in $\Fix(x)$.

By Lemma \ref{fixspan}, the subspace $\Fix(x)$ is spanned by $\widehat{x}$. Hence there exists $\alpha\in \mathbb{R}$ such that 
\begin{equation}
\label{gammasumalpha}
    \sum_{z\in \Gamma(x)}\widehat{z}=\alpha\widehat{x}.
\end{equation} 
It remains to show that $\alpha=\theta_1$.

Referring to (\ref{gammasumalpha}), we take the inner product of each side with $\widehat{x}$ to obtain 
\begin{equation}
    \label{prodx}
    \sum_{z\in \Gamma(x)} \bigl<\widehat{x},\widehat{z}\bigr>=\alpha\bigl\Vert \widehat{x}\bigr\Vert^2.
\end{equation}

By (\ref{valency}), we have $\kappa=\bigl\vert \Gamma(x)\bigr\vert$; refer to (\ref{kappa}) for the value of $\kappa$. In (\ref{prodx}), we evaluate $\bigl<\widehat{x},\widehat{z}\bigr>$ using Lemma \ref{cor1} and $\bigl\Vert \widehat{x}\bigr\Vert ^2$ using Corollary \ref{cor2} and solve the resulting equation for $\alpha$. Evaluate the result and simplify using (\ref{eigenvalues}) to obtain $\alpha=\theta_1$.
\end{proof}

\begin{lemma}
\label{cor3}
The vector space $E$ is spanned by $\bigl\{\widehat{x}\mid x\in X\bigr\}$.
\end{lemma}

\begin{proof}
    In view of (C1), it suffices to show that $\widehat{s}\in \text{Span}\bigl\{\widehat{x}\mid x\in X\bigr\}$ for every $s\in P_1$.  
    Pick $s\in P_1$.
    There exists $y\in X$ that contains $s$. 
    By Lemma \ref{sigmau}, the vector 
    \begin{equation}
    \label{stabsum}
        \sum_{\sigma\in \Stab(s)}\sigma\bigl(\widehat{y}\bigr)
    \end{equation}
    is contained in $\text{Span}\bigl\{\widehat{x}\mid x\in X\bigr\}$. Next we show that (\ref{stabsum}) is a nonzero scalar multiple of $\widehat{s}$. The vector (\ref{stabsum}) is fixed by every element of $\Stab(s)$. Hence, the vector (\ref{stabsum}) is contained in $\Fix(s)$. By Lemma \ref{fixspan}, there exists $\beta\in \mathbb{R}$ such that (\ref{stabsum}) is equal to $\beta\widehat{s}$. 
    We will show that the scalar $\beta$ is nonzero. 
    Take the inner product between (\ref{stabsum}) and $\widehat{s}$.
    Using Lemmas \ref{sigmau}, \ref{introlem1} and the fact that $s\cap \sigma(y)=s$ for all $\sigma\in \Stab(s)$, we get
    \begin{equation*}
        \beta\bigl\Vert \widehat{s}\bigr\Vert^{2} =\sum_{\sigma\in \Stab(s)}\Bigl<\sigma\bigl(\widehat{y}\bigr),\widehat{s}\Bigr>=\sum_{\sigma\in \Stab(s)}\Bigl([n]-[k]\Bigr)=\bigl\vert \Stab(s)\bigr\vert \Bigl([n]-[k]\Bigr)\neq 0.
    \end{equation*}

    Hence $\beta\neq 0$. We have shown that the vector (\ref{stabsum}) is a nonzero scalar multiple of $\widehat{s}$. We mentioned earlier that (\ref{stabsum}) is in $\text{Span}\bigl\{\widehat{x}\mid x\in X\bigr\}$. By these comments, $\widehat{s}\in \text{Span}\bigl\{\widehat{x}\mid x\in X\bigr\}$ and the result follows.
\end{proof}

By Lemmas \ref{cor1}, \ref{introlem2}, \ref{cor3} the Euclidean space $E$, together with the restriction of the map (\ref{hardhat}) to $X$ gives a Euclidean representation of $\Gamma$. This representation is associated with the eigenvalue $\theta_1$. 

We finish this section with a remark. By \cite[Theorem~9.3.3]{BCN}, $GL(V)$ acts on $\Gamma$ in a distance-transitive fashion. Applying \cite[Proposition~4.1.11]{BCN}, we get that the Euclidean space $E$ is irreducible as a $GL(V)$-module.

\section{The stabilizer of two elements in $P$}
\label{stab2}
In this section we pick two distinct elements in $P$ and consider their stabilizer in $GL(V)$. We describe the orbits of this stabilizer acting on $P_1$.

Pick distinct $u,v\in P$ such that $0\neq u\neq V$ and $0\neq v\neq V$. Let $\Stab(u,v)$ denote the subgroup of $GL(V)$ consisting of the elements that fix both $u$ and $v$. We call $\Stab(u,v)$ the {\it stabilizer of $u$ and $v$ in $GL(V)$}. 

We will describe the action of $\Stab(u,v)$ on $P_1$. There are six cases to consider:

\begin{center}
\begin{tabular}{c|c}
    Case & description \\
    \hline
    1 & $u\cap v\neq 0$, $u\not\subseteq v$, $v\not\subseteq u$, $u+v\neq V$\\
    2 & $u\cap v\neq 0$, $u\not\subseteq v$, $v\not\subseteq u$, $u+v=V$\\
    3 & $u\cap v= 0$, $u\not\subseteq v$, $v\not\subseteq u$, $u+v\neq V$\\
    4 & $u\cap v=0$, $u\not\subseteq v$, $v\not\subseteq u$, $u+v=V$\\
    5 & $u\subseteq v$\\
    6 & $v\subseteq u$.
\end{tabular}
\end{center}

\begin{lemma}
\label{main1}
In the table below, we give the orbits for the action of $\Stab(u,v)$ on $P_1$.

{\rm{
\begin{center}
\begin{tabular}{c|c}
    Case & orbits of $\Stab(u,v)$ on $P_1$ \\
    \hline
     & \\
    1 & $\Omega(u\cap v),\qquad \Omega(u)\setminus \Omega(u\cap v),\qquad \Omega (v)\setminus\Omega(u\cap v),\qquad \Omega(u+v)\setminus \bigl(\Omega(u) \cup \Omega(v)\bigr),\qquad P_1\setminus \Omega(u+v)$\\
     & \\
    2 & $\Omega(u\cap v),\qquad 
        \Omega(u)\setminus \Omega(u\cap v),\qquad
        \Omega (v)\setminus \Omega(u\cap v),\qquad P_1\setminus \bigl(\Omega(u) \cup \Omega(v)\bigr)$ \\
         & \\
    3 & $\Omega(u),\qquad \Omega (v),\qquad \Omega(u+v)\setminus \bigl(\Omega(u) \cup \Omega(v)\bigr),\qquad 
        P_1\setminus \Omega(u+v)$\\
         & \\
    4 & $\Omega(u),\qquad \Omega (v),\qquad P_1\setminus \bigl(\Omega(u) \cup \Omega(v)\bigr)$\\
     & \\
    5 & $\Omega(u),\qquad \Omega (v)\setminus \Omega(u),\qquad P_1\setminus \Omega(v)$\\
        & \\
    6 & $\Omega(v),\qquad \Omega (u)\setminus \Omega(v),\qquad P_1\setminus \Omega(u)$
\end{tabular}
\end{center}
}}
\end{lemma}

\begin{proof}
    First assume Case 1. Observe that the five sets displayed in the table are fixed by the elements of $\Stab(u,v)$. Hence each set is a disjoint union of orbits of $\Stab(u,v)$.
    
    We now show that each set forms a single orbit of $\Stab(u,v)$. We start with $\Omega(u\cap v)$.

    Pick distinct $s,s'\in \Omega(u\cap v)$. We show that there exists $\sigma\in \Stab(u,v)$ that sends $s\mapsto s'$.

    Let $\eta$ denote a nonzero vector in $s$. Note that the set $\{\eta\}$ is a basis for $s$; extend this basis to a basis $\{\eta\}\cup \mathcal{R}$ for $u\cap v$. Extend the basis $\{\eta\}\cup \mathcal{R}$ for $u\cap v$ to a basis $\{\eta\}\cup \mathcal{R}\cup \mathcal{T}$ for $u$. Extend the basis $\{\eta\}\cup \mathcal{R}$ for $u\cap v$ to a basis $\{\eta\}\cup \mathcal{R}\cup \mathcal{P}$ for $v$. By Lemma \ref{linin}, $\{\eta\}\cup \mathcal{R}\cup \mathcal{T}\cup \mathcal{P}$ is a basis for $u+v$. Extend the basis $\{\eta\}\cup \mathcal{R}\cup \mathcal{T}\cup \mathcal{P}$ for $u+v$ to a basis $\{\eta\}\cup \mathcal{R}\cup \mathcal{T}\cup \mathcal{P}\cup \mathcal{Q}$ for $V$.
    
    Let $\eta'$ denote a nonzero vector in $s'$. Note that the set $\bigl\{\eta'\bigr\}$ is a basis for $s'$; extend this basis to a basis $\bigl\{\eta'\bigr\}\cup \mathcal{R}'$ for $u\cap v$. Note that $\bigl\{\eta'\bigr\}\cup \mathcal{R}'\cup \mathcal{T}$ is a basis for $u$. Note that $\bigl\{\eta'\bigr\}\cup \mathcal{R}'\cup \mathcal{P}$ is a basis for $v$. By Lemma \ref{linin}, $\bigl\{\eta'\bigr\}\cup \mathcal{R}'\cup \mathcal{T}\cup \mathcal{P}$ is a basis for $u+v$. Note that $\bigl\{\eta'\bigr\}\cup \mathcal{R}'\cup \mathcal{T}\cup \mathcal{P}\cup \mathcal{Q}$ is a basis for $V$.

    By linear algebra, there exists $\sigma\in GL(V)$ that sends $\eta\mapsto \eta'$, $\mathcal{R}\mapsto \mathcal{R}'$ and acts as the identity on each of $\mathcal{T}$, $\mathcal{P}$, $\mathcal{Q}$. By construction, $\sigma$ is contained in $\Stab(u,v)$ and sends $s\mapsto s'$.

    Next we show that the set $\Omega(u)\setminus \Omega(u\cap v)$ is a single orbit of $\Stab(u,v)$.

    Pick distinct $s,s'\in \Omega(u)\setminus \Omega(u\cap v)$. We show that there exists $\sigma\in \Stab(u,v)$ that sends $s\mapsto s'$.
    
    Let $\mathcal{R}\subseteq V$ form a basis for $u\cap v$. Let $\eta$ denote a nonzero vector in $s$. Note that $\{\eta\}$ is a basis for $s$. The subspace $s$ is not contained in $u \cap v$, so $\{\eta\} \cup \mathcal{R}$ is a basis for $s + (u \cap v)$. Extend the basis $\{\eta\} \cup \mathcal{R}$ for $s + (u \cap v)$ to a basis $\{\eta\}\cup \mathcal{R}\cup \mathcal{T}$ for $u$. Extend the basis $\mathcal{R}$ for $u\cap v$ to a basis $\mathcal{R}\cup \mathcal{P}$ for $v$. By Lemma \ref{linin}, $\{\eta\}\cup \mathcal{R}\cup \mathcal{T}\cup \mathcal{P}$ is a basis for $u+v$. Extend the basis $\{\eta\}\cup \mathcal{R}\cup \mathcal{T}\cup \mathcal{P}$ for $u+v$ to a basis $\{\eta\}\cup \mathcal{R}\cup \mathcal{T}\cup \mathcal{P}\cup \mathcal{Q}$ for $V$.

    Let $\eta'$ denote a nonzero vector in $s'$. Note that $\bigl\{\eta'\bigr\}$ is a basis for $s'$. The subspace $s'$ is not contained in $u \cap v$, so $\bigl\{\eta'\bigr\} \cup \mathcal{R}$ is a basis for $s' + (u \cap v)$. Extend the basis $\bigl\{\eta'\bigr\} \cup \mathcal{R}$ for $s' + (u \cap v)$ to a basis $\bigl\{\eta'\bigr\}\cup \mathcal{R}\cup \mathcal{T}'$ for $u$. By Lemma \ref{linin}, $\bigl\{\eta'\bigr\}\cup \mathcal{R}\cup \mathcal{T}'\cup \mathcal{P}$ is a basis for $u+v$. Note that $\bigl\{\eta'\bigr\}\cup \mathcal{R}\cup \mathcal{T}'\cup \mathcal{P}\cup \mathcal{Q}$ is a basis for $V$.

    By linear algebra, there exists $\sigma\in GL(V)$ that sends $\eta\mapsto \eta'$, $\mathcal{T}\mapsto \mathcal{T}'$ and acts as the identity on each of $\mathcal{R}$, $\mathcal{P}$, $\mathcal{Q}$. By construction, $\sigma$ is contained in $\Stab(u,v)$ and sends $s\mapsto s'$.

    By symmetry, the set $\Omega(v)\setminus \Omega(u\cap v)$ is a single orbit of $\Stab(u,v)$.

    Next we show that the set $\Omega(u+v)\setminus \bigl(\Omega(u)\cup \Omega(v)\bigr)$ is a single orbit of $\Stab(u,v)$. 

    Pick distinct $s,s'\in \Omega(u+v)\setminus \bigl(\Omega(u)\cup \Omega(v)\bigr)$. We show that there exists $\sigma\in \Stab(u,v)$ that sends $s\mapsto s'$.

    Let $\eta$ denote a nonzero vector in $s$ and let $\eta'$ denote a nonzero vector in $s'$. By linear algebra, there exist sets $\mathcal{R},\mathcal{T},\mathcal{P}$ such that $\mathcal{R}$ is a basis for $u\cap v$, $\mathcal{R}\cup \mathcal{T}$ is a basis for $u$, $\mathcal{R}\cup \mathcal{P}$ is a basis for $v$, and $\eta=r+t+p$ for $r\in \mathcal{R}$, $t\in \mathcal{T}$, $p\in \mathcal{P}$. 
    
    Similarly, there exist sets $\mathcal{R}',\mathcal{T}',\mathcal{P}'$ such that $\mathcal{R}'$ is a basis for $u\cap v$, $\mathcal{R}'\cup \mathcal{T}'$ is a basis for $u$, $\mathcal{R}'\cup \mathcal{P}'$ is a basis for $v$, and $\eta'=r'+t'+p'$ for $r'\in \mathcal{R}'$, $t'\in \mathcal{T}'$, $p'\in \mathcal{P}'$. 

    By linear algebra, there exists $\sigma\in GL(V)$ that sends $\mathcal{R}\mapsto \mathcal{R}'$, $\mathcal{T}\mapsto \mathcal{T}'$, $\mathcal{P}\mapsto \mathcal{P}'$ such that $r\mapsto r'$, $t\mapsto t'$, $p\mapsto p'$. By construction, $\sigma$ is contained in $\Stab(u,v)$ and sends $s\mapsto s'$.

    Next we show that the set $P_1\setminus \Omega(u+v)$ is a single orbit of $\Stab(u,v)$.

    Pick distinct $s,s'\in P_1\setminus \Omega(u+v)$. We show that there exists $\sigma\in \Stab(u,v)$ that sends $s\mapsto s'$.

    Let $\mathcal{P}\subseteq V$ form a basis for $u+v$. Let $\eta$ denote a nonzero vector in $s$. Note that $\{\eta\}$ is a basis for $s$. The subspace $s$ is not contained in $u+v$, so $\{\eta\}\cup\mathcal{P}$ is a basis for $s+u+v$. Extend the basis $\{\eta\}\cup\mathcal{P}$ for $s+u+v$ to a basis $\{\eta\}\cup\mathcal{P}\cup \mathcal{Q}$ for $V$. 

    Let $\eta'$ denote a nonzero vector in $s'$. Note that $\bigl\{\eta'\bigr\}$ is a basis for $s'$. The subspace $s'$ is not contained in $u+v$, so $\bigl\{\eta'\bigr\}\cup\mathcal{P}$ is a basis for $s'+u+v$. Extend the basis $\bigl\{\eta'\bigr\}\cup\mathcal{P}$ for $s'+u+v$ to a basis $\bigl\{\eta'\bigr\}\cup\mathcal{P}\cup \mathcal{Q}'$ for $V$. 

    By linear algebra, there exists $\sigma\in GL(V)$ that sends $\eta\mapsto \eta'$, $\mathcal{Q}\mapsto \mathcal{Q}'$ and acts as the identity on $\mathcal{P}$. By construction, $\sigma$ is contained in $\Stab(u,v)$ and sends $s\mapsto s'$.

    We have proved the result for Case 1. For the remaining cases the proof is similar, and omitted.
\end{proof}

\section{The stabilizer of two elements in $P$; action on $E$}
\label{actionone}

In this section we pick two distinct elements in $P$ and consider their stabilizer in $GL(V)$. We consider the action of this stabilizer on the Euclidean space $E$.

Pick distinct $u,v\in P$ such that $0\neq u\neq V$ and $0\neq v\neq V$. Let $\Fix(u,v)$ denote the subspace of $E$ consisting of the vectors that are fixed by every element of $\Stab(u,v)$. We give a basis for $\Fix(u,v)$. In what follows, we refer to the cases in Lemma \ref{main1}.

\begin{lemma}
    \label{main2}
    In the table below, we display a basis for $\Fix(u,v)$. 
{\rm{
    \begin{center}
        \begin{tabular}{c|c}
            Case & basis for $\Fix(u,v)$  \\
            \hline
             & \\
            1 & \;$\widehat{u\cap v}, \qquad \qquad \widehat{u},\qquad \qquad  \widehat{v}, \qquad \qquad  \widehat{u+v}$\\
            & \\
            2 & $\widehat{u\cap v}, \qquad \qquad \widehat{u},\qquad \qquad \widehat{v}$\\
            & \\
            3 & $\widehat{u}, \qquad \qquad \widehat{v}, \qquad \qquad \widehat{u+v}$\\
            & \\
            4 & $\widehat{u}, \qquad \qquad \widehat{v}$\\
            & \\
            5 & $\widehat{u}, \qquad \qquad \widehat{v}$\\
            & \\
            6 & $\widehat{u}, \qquad \qquad \widehat{v}$\\
        \end{tabular}
    \end{center}
}}
\end{lemma}

\begin{proof}
    We first assume Case 1. By Lemma \ref{character}, the subspace $\Fix(u,v)$ is spanned by the characteristic vectors of the $\Stab(u,v)$-orbits on $P_1$. These orbits are given in Lemma \ref{main1}. Their characteristic vectors are
    \begin{equation}
    \label{orbvec1}
        \widehat{u\cap v}, \qquad \widehat{u}-\widehat{u\cap v}, \qquad \widehat{v}-\widehat{u\cap v}, \qquad \widehat{u+v}-\widehat{u}-\widehat{v}+\widehat{u\cap v}, \qquad -\widehat{u+v}.
    \end{equation}
    The five vectors in (\ref{orbvec1}) sum to $0$. By Lemma \ref{sum0converse}, any four of these vectors are linearly independent, and hence form a basis for $\Fix(u,v)$. In particular, the following vectors form a basis for $\Fix(u,v)$:
    \begin{equation*}
        \widehat{u\cap v}, \qquad \widehat{u}-\widehat{u\cap v}, \qquad \widehat{v}-\widehat{u\cap v}, \qquad -\widehat{u+v}.
    \end{equation*}
    
    Adjusting the basis above, the result follows.

    We have proved the result for Case 1. For the remaining cases the proof is similar, and omitted.
\end{proof}

\begin{definition}
\label{geometricdef}
    Pick distinct $u,v\in P$. By the \emph{geometric basis for $\Fix(u,v)$}, we mean the basis displayed in Lemma \ref{main2}.
\end{definition} 

In the next result, we consider how Lemma \ref{main2} looks for the case in which $u,v \in X$.

\begin{theorem}
    \label{1basis}
    Pick distinct $x,y\in X$. In the table below, we display the geometric basis for $\Fix(x,y)$. 
{\rm{
    \begin{center}
        \begin{tabular}{c|c}
            Case & geometric basis for $\Fix(x,y)$  \\
            \hline
             & \\
            $1\leq \partial(x,y)<k$\; & \;$\widehat{x},\qquad \widehat{y}, \qquad \widehat{x\cap y}, \qquad \widehat{x+y}$\\
            & \\
            $\partial(x,y)=k$\; & \;$\widehat{x},\qquad \widehat{y}, \qquad \widehat{x+y}$\\
        \end{tabular}
    \end{center}
}}
\end{theorem}

\begin{proof}
    Since $x$ and $y$ are distinct, $x\not\subseteq y$ and $y\not\subseteq x$. Since $n>2k$, we have $x+y\neq V$.
    
    We first assume that $1\leq \partial(x,y)<k$. Since $\partial(x,y)<k$, we have $x\cap y\neq 0$. The result follows from Case 1 of Lemma \ref{main2}.

    Next we assume that $\partial(x,y)=k$. By this assumption, $x\cap y=0$. The result follows from Case 3 of Lemma \ref{main2}.
\end{proof}

Pick distinct $x,y\in X$. Our next general goal is to use the graph $\Gamma$ to find another basis for $\Fix(x,y)$, called the combinatorial basis. We will focus on the case $1<\partial(x,y)<k$. The cases $\partial(x,y)=1$ and $\partial(x,y)=k$ are more involved and require different methods; these cases will be handled in a future paper.

\section{The subspace $\Fix(x,y)$ and its combinatorial basis}
Pick $x,y\in X$ such that $1<\partial(x,y)<k$. In this section we use the graph $\Gamma$ to construct two vectors $B_{xy}$ and $C_{xy}$. In Section \ref{more}, we will show that the following vectors form a basis for $\Fix(x,y)$:
\begin{equation}
\label{basis2}
    \widehat{x},\qquad \qquad  \widehat{y}, \qquad \qquad B_{xy},\qquad \qquad  C_{xy}.
\end{equation}
The basis (\ref{basis2}) will be called the combinatorial basis. We will formally define this basis in Section \ref{more}.

To define the vectors $B_{xy}$ and $C_{xy}$, we first consider two orbits of the $\Stab(x,y)$-action on $P$.

\begin{definition}
\label{calbcdef}
For $x,y\in X$ such that $1<\partial(x,y)<k$, define 
\begin{align*}
    \mathcal{B}_{xy}&=\{z\in \Gamma(x)\mid \partial(y,z)=\partial(x,y)+1\},\\
    \mathcal{C}_{xy}&=\{z\in \Gamma(x)\mid \partial(y,z)=\partial(x,y)-1\}.
\end{align*}
\end{definition}

Observe that $|\mathcal{B}_{xy}|=b_i$ and $|\mathcal{C}_{xy}|=c_i$, where $i=\partial(x,y)$.

Our next goal is to show that the sets $\mathcal{B}_{xy}$ and $\mathcal{C}_{xy}$ are orbits of $\Stab(x,y)$.

\begin{lemma}
\label{orbitcxy}
    For $x,y\in X$ such that $1<\partial(x,y)<k$, the set $\mathcal{C}_{xy}$ is an orbit of the $\Stab(x,y)$-action on $P$.
\end{lemma}

\begin{proof}
    By Lemma \ref{presdist}, the set $\mathcal{C}_{xy}$ is a disjoint union of orbits of $\Stab(x,y)$. We now show that this union consists of a single orbit. Let $z,z'\in \mathcal{C}_{xy}$. We display an element $\sigma\in \Stab(x,y)$ that sends $z\mapsto z'$.

    By Lemma \ref{contain}, 
    \begin{equation*}
        x\cap y\subseteq z\subseteq x+y, \qquad \qquad x\cap y\subseteq z'\subseteq x+y.
    \end{equation*}
    
    Let the set $\mathcal{R}\subseteq V$ form a basis for $x\cap y$. Extend the basis $\mathcal{R}$ for $x\cap y$ to a basis $\mathcal{R}\cup \mathcal{S}$ for $x\cap z$.  Extend the basis $\mathcal{R}$ for $x\cap y$ to a basis $\mathcal{R}\cup \mathcal{T}$ for $z\cap y$. By Lemma \ref{linin}, $\mathcal{R}\cup \mathcal{S}\cup \mathcal{T}$ is a basis for $z$. Extend the basis $\mathcal{R}\cup \mathcal{S}$ for $x\cap z$ to a basis $\mathcal{R}\cup \mathcal{S}\cup \mathcal{P}$ for $x$. Extend the basis $\mathcal{R}\cup \mathcal{T}$ for $z\cap y$ to a basis $\mathcal{R}\cup \mathcal{T}\cup \mathcal{Q}$ for $y$. By Lemma \ref{linin}, $\mathcal{R}\cup \mathcal{S}\cup \mathcal{T}\cup \mathcal{P}\cup \mathcal{Q}$ is a basis for $x+y$. Extend the basis $\mathcal{R}\cup \mathcal{S}\cup \mathcal{T}\cup \mathcal{P}\cup \mathcal{Q}$ for $x+y$ to a basis $\mathcal{R}\cup \mathcal{S}\cup \mathcal{T}\cup \mathcal{P}\cup \mathcal{Q}\cup \mathcal{W}$ for $V$.
    
    Extend the basis $\mathcal{R}$ for $x\cap y$ to a basis $\mathcal{R}\cup \mathcal{S}'$ for $x\cap z'$.
    Extend the basis $\mathcal{R}$ for $x\cap y$ to a basis $\mathcal{R}\cup \mathcal{T}'$ for $z'\cap y$. 
    By Lemma \ref{linin}, $\mathcal{R}\cup \mathcal{S}'\cup \mathcal{T}'$ is a basis for $z'$. Extend the basis $\mathcal{R}\cup \mathcal{S}'$ for $x\cap z'$ to a basis $\mathcal{R}\cup \mathcal{S}'\cup \mathcal{P}'$ for $x$. Extend the basis $R\cup \mathcal{T}'$ for $z'\cap y$ to a basis $\mathcal{R}\cup \mathcal{T}'\cup \mathcal{Q}'$ for $y$. By Lemma \ref{linin}, $\mathcal{R}\cup \mathcal{S}'\cup \mathcal{T}'\cup \mathcal{P}'\cup \mathcal{Q}'$ is a basis for $x+y$.  Note that $\mathcal{R}\cup \mathcal{S}'\cup \mathcal{T}'\cup \mathcal{P}'\cup \mathcal{Q}'\cup \mathcal{W}$ is a basis for $V$.
    
    By linear algebra, there exists $\sigma\in GL(V)$ that sends $\mathcal{S}\mapsto \mathcal{S}'$, $\mathcal{T}\mapsto \mathcal{T}'$, $\mathcal{P}\mapsto \mathcal{P}'$, $\mathcal{Q}\mapsto \mathcal{Q}'$ and acts as the identity on each of $\mathcal{R}$, $\mathcal{W}$. By construction, $\sigma$ is contained in $\Stab(x,y)$ and sends $z\mapsto z'$. The result follows.
    \end{proof}
\begin{lemma}
\label{geodesic}
    Let $x,y,x',y'\in X$ satisfy $1< \partial(x,y)=\partial\left(x', y'\right)< k$. Let $z\in \mathcal{C}_{xy}$ and $z'\in \mathcal{C}_{x'y'}$. Then there exists an element in $GL(V)$ that sends $x\mapsto x'$, $y\mapsto y'$, $z\mapsto z'$.
\end{lemma}

\begin{proof}
Since $\Gamma$ is distance-transitive, we may assume without loss that $x=x'$ and $y=y'$. The result follows from Lemma \ref{orbitcxy}.
\end{proof}

\begin{lemma}
\label{orbitbxy}
    For $x,y\in X$ such that $1< \partial(x,y)<k$, the set $\mathcal{B}_{xy}$ is an orbit of the $\Stab(x,y)$-action on $P$.  
\end{lemma}

\begin{proof}
     By Lemma \ref{presdist}, the set $\mathcal{B}_{xy}$ is a disjoint union of orbits of $\Stab(x,y)$. We now show that this union consists of a single orbit. Let $z,z'\in \mathcal{B}_{xy}$. By construction, $x\in \mathcal{C}_{zy}$ and $x\in \mathcal{C}_{z'y}$. By Lemma \ref{geodesic}, there exists an element in $\Stab(x,y)$ that sends $z\mapsto z'$. The result follows.
\end{proof}

\begin{definition}
\label{bcdef}
For $x,y\in X$ such that $1<\partial(x,y)<k$, define the vectors
\begin{equation}
\nonumber
    B_{xy}=\sum_{z\in \mathcal{B}_{xy}}\widehat{z},\qquad \qquad C_{xy}=\sum_{z\in \mathcal{C}_{xy}}\widehat{z}.
\end{equation}
\end{definition}

\begin{lemma}
    \label{fixcontain2}
    For $x,y\in X$ such that $1<\partial(x,y)<k$, the subspace $\Fix(x,y)$ contains $B_{xy}$ and $C_{xy}$.
\end{lemma}

\begin{proof}
        Pick $\sigma\in \Stab(x,y)$. The map $\sigma$ fixes the sets $\mathcal{B}_{xy}$ and $\mathcal{C}_{xy}$. The result follows. 
\end{proof}

\section{The inner products involving the geometric basis and the combinatorial basis}
\label{innerproducts}
Pick $x,y\in X$ such that $1<\partial(x,y)<k$. In this section we compute the inner products between:
\begin{enumerate}[label=(\roman*)]
    \item any two vectors in the geometric basis for $\Fix(x,y)$;

    \item any vector in the geometric basis for $\Fix(x,y)$ and any vector in the set (\ref{basis2});

    \item any two vectors in the set (\ref{basis2}).
\end{enumerate}

We start with case (i).

\begin{lemma}
\label{5inner1}
For $x,y\in X$ such that $1<\partial(x,y)<k$, we have 
\begin{equation*}
    \begin{aligned}
    \Bigl<\widehat{x},\widehat{x\cap y}\Bigr>&=q^k[k-i][n-k],\\
    \Bigl<\widehat{y},\widehat{x\cap y}\Bigr>&=q^k[k-i][n-k],
\end{aligned}
\end{equation*}
where $i=\partial(x,y)$.
\end{lemma}

\begin{proof}
Routine using Lemmas \ref{introlem0}, \ref{introlem1} and linear algebra.
\end{proof}

\begin{lemma}
\label{5inner2}
For $x,y\in X$ such that $1<\partial(x,y)<k$, we have
\begin{equation*}
    \begin{aligned}
        \Bigl<\widehat{x},\widehat{x+y}\Bigr>&=q^{k+i}[k][n-k-i], \label{3.1.3}\\
        \Bigl<\widehat{y},\widehat{x+y}\Bigr>&=q^{k+i}[k][n-k-i],
    \end{aligned}
\end{equation*}
where $i=\partial(x,y)$.
\end{lemma}

\begin{proof}
    Routine using Lemmas \ref{sumdim}, \ref{introlem1} and linear algebra.
\end{proof}

\begin{lemma}
\label{5inner3}
For $x,y\in X$ such that $1<\partial(x,y)<k$, we have
\begin{equation}
    \begin{aligned}
        \nonumber
        \Bigl\Vert\widehat{x\cap y}\Bigr\Vert^2&=q^{k-i}[k-i][n-k+i],\\
        \Bigl<\widehat{x\cap y},\widehat{x+y}\Bigr>&=q^{k+i}[k-i][n-k-i],\\
        \Bigl\Vert\widehat{x+y}\Bigr\Vert^2&=q^{k+i}[k+i][n-k-i],
    \end{aligned}
\end{equation}
where $i=\partial(x,y)$.
\end{lemma}
\begin{proof}
Routine using Lemmas \ref{introlem0}, \ref{sumdim}, \ref{introlem1} and linear algebra.
\end{proof}

\begin{theorem}
\label{innprodmat21}
Let $x,y\in X$ satisfy $1<\partial(x,y)<k$. In the following table, for each vector $u$ in the header column, and each vector $v$ in the header row, the $(u,v)$-entry of the table gives the inner product $\left<u,v\right>$. Write $i=\partial(x,y)$.

\begin{center}
    \begin{tabular}{c|c c c c}
        $\left<\;,\;\right>$ & $\widehat{x}$ & $\widehat{y}$ & $\widehat{x\cap y}$ & $\widehat{x+y}$\\ 
         \hline
         \\
         $\widehat{x}$ & $q^k[k][n-k]$ & $[n][k-i]-[k]^2$ & $q^k[k-i][n-k]$ & $q^{k+i}[k][n-k-i]$\\
         \\
        $\widehat{y}$ & $[n][k-i]-[k]^2$ & $q^{k}[k][n-k]$ & $q^k[k-i][n-k]$ & $q^{k+i}[k][n-k-i]$\\
        \\
        $\widehat{x\cap y}$&$q^k[k-i][n-k]$ & $q^k[k-i][n-k]$ & $q^{k-i}[k-i][n-k+i]$ & $q^{k+i}[k-i][n-k-i]$\\
        \\
        $\widehat{x+y}$&$q^{k+i}[k][n-k-i]$ & $q^{k+i}[k][n-k-i]$ & $q^{k+i}[k-i][n-k-i]$ & $q^{k+i}[k+i][n-k-i]$\\
        
    \end{tabular}
\end{center}
\end{theorem}

\begin{proof}
Combine Corollary \ref{cor2} and Lemmas \ref{cor1}, \ref{5inner1}--\ref{5inner3}.
\end{proof}

Next we compute the inner products for case (ii).

\begin{lemma}
\label{bcx}
For $x,y\in X$ such that $1<\partial(x,y)<k$, we have 
\begin{align}
    \bigl<B_{xy},\widehat{x}\bigr>&=q^{2i+1}[k-i][n-k-i]\Bigl([n][k-1]-[k]^2\Bigr),\label{binnerx}\\
    \bigl<C_{xy},\widehat{x}\bigr>&=[i]^2\Bigl([n][k-1]-[k]^2\Bigr), \label{cinnerx}
\end{align}
where $i=\partial(x,y)$.
\end{lemma}
\begin{proof}
We first prove (\ref{binnerx}). In the left equation of Definition \ref{bcdef}, take the inner product of each side with $\widehat{x}$. Evaluate the result using Lemma \ref{cor1} and (\ref{sizebc}) to get (\ref{binnerx}).

We have now verified (\ref{binnerx}). (\ref{cinnerx}) is obtained in a similar fashion.
\end{proof}

\begin{lemma}
\label{bcy}
For $x,y\in X$ such that $1<\partial(x,y)<k$, we have
\begin{align}
    \bigl<B_{xy},\widehat{y}\bigr>&=q^{2i+1}[k-i][n-k-i]\Bigl([n][k-i-1]-[k]^2\Bigr),\label{binnery}\\
    \bigl<C_{xy},\widehat{y}\bigr>&=[i]^2\Bigl([n][k-i+1]-[k]^2\Bigr),\label{cinnery}
\end{align}
where $i=\partial(x,y)$.
\end{lemma}
\begin{proof}
We first prove (\ref{binnery}). In the left equation of Definition \ref{bcdef}, take the inner product of each side with $\widehat{y}$. Evaluate the result using Lemma \ref{cor1} and (\ref{sizebc}) to get (\ref{binnery}).

We have now verified (\ref{binnery}). (\ref{cinnery}) is obtained in a similar fashion.
\end{proof}

\begin{lemma}
\label{bintsum}
For $x,y\in X$ such that $1<\partial(x,y)<k$, we have
\begin{align}
    \Bigl<B_{xy},\widehat{x\cap y}\Bigr>&=q^{2i+1}[k-i][n-k-i]\Bigl([n][k-i-1]-[k-i][k]\Bigr),\label{binnercap}\\
    \Bigl<B_{xy},\widehat{x+y}\Bigr>&=q^{2i+1}[k-i][n-k-i]\Bigl([n][k-1]-[k][k+i]\Bigr),\label{binnerplus}
\end{align}
where $i=\partial(x,y)$.
\end{lemma}

\begin{proof} 
We first prove (\ref{binnercap}). Using the left equation of Definition \ref{bcdef}, we obtain
\begin{equation*}
    \Bigl<B_{xy},\widehat{x\cap y}\Bigr>=\sum_{z\in \mathcal{B}_{xy}}\Bigl<\widehat{z},\widehat{x\cap y}\Bigr>.
\end{equation*}

Let $z\in \mathcal{B}_{xy}$. By the definition of $\mathcal{B}_{xy}$, we have $\partial(z,y)=i+1$. By Lemma \ref{introlem0}, $\dim(z\cap y)=k-i-1$. Also, $\partial(z,y)=\partial(z,x)+\partial(x,y)$, so by Lemma \ref{contain}, $x\supseteq z\cap y$. Hence  
\begin{equation*}
    \dim(z\cap x\cap y)=\dim (z\cap y)=k-i-1.
\end{equation*}
Using Lemma \ref{introlem1} we obtain
\begin{equation*}
    \Bigl<\widehat{z},\widehat{x\cap y}\Bigr>=[n][k-i-1]-[k-i][k].
\end{equation*}

Use the value of $b_i$ in (\ref{sizebc}) to obtain (\ref{binnercap}).

Next we prove (\ref{binnerplus}). Using the left equation of Definition \ref{bcdef}, we obtain
\begin{equation*}
    \Bigl<B_{xy},\widehat{x+y}\Bigr>=\sum_{z\in \mathcal{B}_{xy}}\Bigl<\widehat{z},\widehat{x+y}\Bigr>.
\end{equation*}

Let $z\in \mathcal{B}_{xy}$. Recall that $\partial(z,y)=i+1$. By Lemma \ref{sumdim}, $\dim(z+y)=k+i+1$. Also recall that $\partial(z,y)=\partial(z,x)+\partial(x,y)$, so by Lemma \ref{contain}, $x\subseteq z+y$. Hence  
\begin{equation*}
    \dim(z+x+y)=\dim (z+y)=k+i+1.
\end{equation*}
By Lemma \ref{modularity},
\begin{equation*}
    \dim\bigl(z\cap (x+y)\bigr)=\dim z +\dim(x+y)-\dim (z+x+y)=k-1.
\end{equation*}

Using Lemma \ref{introlem1} we obtain
\begin{equation*}
    \Bigl<\widehat{z},\widehat{x+y}\Bigr>=[n][k-1]-[k][k+i].
\end{equation*}

Use the value of $b_i$ in (\ref{sizebc}) to obtain (\ref{binnerplus}).
\end{proof}

\begin{lemma}
\label{cintsum}
For $x,y\in X$ such that $1<\partial(x,y)<k$, we have
\begin{align}
    \Bigl<C_{xy},\widehat{x\cap y}\Bigr>&=q^k[i]^2[k-i][n-k],\label{cinnercap}\\
    \Bigl<C_{xy},\widehat{x+y}\Bigr>&=q^{k+i}[i]^2[k][n-k-i],\label{cinnerplus}
\end{align}
where $i=\partial(x,y)$.
\end{lemma}
\begin{proof}
We first prove (\ref{cinnercap}). Using the right equation of Definition \ref{bcdef}, we obtain
\begin{equation*}
    \Bigl<C_{xy},\widehat{x\cap y}\Bigr>=\sum_{z\in \mathcal{C}_{xy}}\Bigl<\widehat{z},\widehat{x\cap y}\Bigr>.
\end{equation*}

Let $z\in \mathcal{C}_{xy}$. By the definition of $\mathcal{C}_{xy}$, we have $\partial(z,y)=i-1$. Since $\partial(x,y)=\partial(x,z)+\partial(z,y)$, we have $z\supseteq x\cap y$ by Lemma \ref{contain}. Hence  
\begin{equation*}
    \dim(z\cap x\cap y)=\dim (x\cap y)=k-i.
\end{equation*}
Using Lemma \ref{introlem1} we obtain
\begin{equation*}
    \Bigl<\widehat{z},\widehat{x\cap y}\Bigr>=[n][k-i]-[k][k-i]=q^k[k-i][n-k].
\end{equation*}

Use the value of $c_i$ in (\ref{sizebc}) to obtain (\ref{cinnercap}).

Next we prove (\ref{cinnerplus}). Using the right equation of Definition \ref{bcdef}, we obtain
\begin{equation*}
    \Bigl<C_{xy},\widehat{x+y}\Bigr>=\sum_{z\in \mathcal{C}_{xy}}\Bigl<\widehat{z},\widehat{x+y}\Bigr>.
\end{equation*}

Let $z\in \mathcal{C}_{xy}$. Recall that $\partial(x,y)=\partial(x,z)+\partial(z,y)$, so by Lemma \ref{contain}, $z\subseteq x+y$. Hence  
\begin{equation*}
    z\cap (x+y)=z.
\end{equation*}
Therefore,
\begin{equation*}
    \dim\bigl(z\cap (x+y)\bigr)=\dim z=k.
\end{equation*}

Using Lemma \ref{introlem1} we obtain
\begin{equation*}
    \Bigl<\widehat{z},\widehat{x+y}\Bigr>=[n][k]-[k][k+i]=q^{k+i}[k][n-k-i].
\end{equation*}

Use the value of $c_i$ in (\ref{sizebc}) to obtain (\ref{cinnerplus}).
\end{proof}

\begin{theorem}
\label{innprodmat23}
Let $x,y\in X$ satisfy $1<\partial(x,y)<k$. In the following table, for each vector $u$ in the header column, and each vector $v$ in the header row, the $(u,v)$-entry of the table gives the inner product $\left<u,v\right>$. Write $i=\partial(x,y)$.

\begin{center}
    \begin{tabular}{c|c c c c}
        $\left<\;,\;\right>$ & $\widehat{x}$ & $\widehat{y}$ & $\widehat{x\cap y}$ & $\widehat{x+y}$\\ 
         \hline
         \\
        $\widehat{x}$ & $q^k[k][n-k]$ & $[n][k-i]-[k]^2$ & $q^k[k-i][n-k]$ & $q^{k+i}[k][n-k-i]$ \\
        \\
        $\widehat{y}$ & $[n][k-i]-[k]^2$ & $q^k[k][n-k]$ & $q^k[k-i][n-k]$ & $q^{k+i}[k][n-k-i]$ \\
        \\
        $B_{xy}$& $\substack{q^{2i+1}[k-i][n-k-i]\cdot\\ \left([n][k-1]-[k]^2\right)}$ & $\substack{q^{2i+1}[k-i][n-k-i]\cdot\\ \left([n][k-i-1]-[k]^2\right)}$ & $\substack{q^{2i+1}[k-i][n-k-i]\cdot\\ \left([n][k-i-1]-[k-i][k]\right)}$ & $\substack{q^{2i+1}[k-i][n-k-i]\cdot\\ \left([n][k-1]-[k][k+i]\right)}$ \\
        \\
        $C_{xy}$& $[i]^2\bigl([n][k-1]-[k]^2\bigr)$& $[i]^2\bigl([n][k-i+1]-[k]^2\bigr)$ & $q^k[i]^2[k-i][n-k]$ & $q^{k+i}[i]^2[k][n-k-i]$\\
        
    \end{tabular}
\end{center}
\end{theorem}
\begin{proof}
Combine Corollary \ref{cor2} and Lemmas \ref{cor1}, \ref{5inner1}, \ref{5inner2}, \ref{bcx}--\ref{cintsum}.
\end{proof}

Next we compute the inner products for case (iii). To do this, we first write the vectors $B_{xy}$ and $C_{xy}$ as linear combinations in the geometric basis for $\Fix(x,y)$. 

For $1<i<k$ let $M_i$ denote the matrix of inner products in Theorem \ref{innprodmat21}.

\begin{lemma}
\label{inverse}
For $1<i<k$ the inverse of the matrix $M_i$ is given by 
\begin{equation}
    \label{invmat}
    M_{i}^{-1}=\frac{1}{q^{k-i}(q-1)[i]^2[n]}\begin{pmatrix}
q^i&1&-q^i&-1\\
1 & q^i & -q^i & -1\\
-q^i & -q^i & \frac{q^i[k]-[i]}{[k-i]} & 1\\
-1 & -1 & 1 & \frac{q^i[n-k]-[i]}{q^{2i}[n-k-i]}
\end{pmatrix}.
\end{equation}
\end{lemma}
\begin{proof}
Routine.
\end{proof}

\begin{lemma}
\label{bc}
For $x,y\in X$ such that $1<\partial(x,y)<k$, we have
    \begin{align}
    B_{xy}&=q^{2i}[k-i][n-k-i]\widehat{x}-q^{2i}[n-k-i]\widehat{x\cap y}-q^i[k-i]\widehat{x+y},\label{blin}\\
    C_{xy}&=q[i-1]^2\widehat{x}+q^{i-1}\widehat{y}+q^{i}[i-1]\widehat{x\cap y}+[i-1]\widehat{x+y},\label{clin}
    \end{align}
where $i=\partial(x,y)$.
\end{lemma}
\begin{proof}
Write 
    \begin{align}
    B_{xy}&=\alpha\widehat{x}+\beta\widehat{y}+\gamma\widehat{x\cap y}+\delta\widehat{x+y},\label{blinwrite}\\
        C_{xy}&=\alpha'\widehat{x}+\beta'\widehat{y}+\gamma'\widehat{x\cap y}+\delta'\widehat{x+y},\label{clinwrite}
    \end{align}
for $\alpha,\beta,\gamma,\delta,\alpha',\beta',\gamma',\delta'\in \mathbb{R}$. 

In each of (\ref{blinwrite}) and (\ref{clinwrite}), we take the inner product of either side with each of $\widehat{x}, \widehat{y}, \widehat{x\cap y}$, $\widehat{x+y}$ to obtain 
\begin{equation}
        \nonumber
        M_i\begin{pmatrix}
\alpha&\alpha'\\ \beta&\beta'\\ \gamma&\gamma'\\ \delta&\delta'
\end{pmatrix}=\begin{pmatrix}
\left<B_{xy},\widehat{x}\right>&\left<C_{xy},\widehat{x}\right>\\
\left<B_{xy},\widehat{y}\right>&\left<C_{xy},\widehat{y}\right>\\
\left<B_{xy},\widehat{x\cap y}\right>&\left<C_{xy},\widehat{x\cap y}\right>\\
\left<B_{xy},\widehat{x+y}\right>&\left<C_{xy},\widehat{x+y}\right>
\end{pmatrix}.
\end{equation}

Use the table in Theorem \ref{innprodmat23} and the matrix (\ref{invmat}) to obtain
\begin{equation}
\label{coeff}
        \begin{pmatrix}
\alpha&\alpha'\\ \beta&\beta'\\ \gamma&\gamma'\\ \delta&\delta'
\end{pmatrix}
=
\begin{pmatrix}
q^{2i}[k-i][n-k-i]&q[i-1]^2\\
0&q^{i-1}\\
-q^{2i}[n-k-i]&q^i[i-1]\\
-q^i[k-i]&[i-1]\\
\end{pmatrix}.
\end{equation}
The result follows.
\end{proof}

\begin{lemma}
\label{innbb}
For $x,y\in X$ such that $1<\partial(x,y)<k$, we have
\begin{equation}
    \begin{aligned}
    \nonumber
        \bigl<B_{xy},B_{xy}\bigr>=&\;q^{4i+2}[k-i][n-k-i]\biggl(q^{k-i-2}[n]\Bigl([k-i]+[n-k-i]\Bigr)\\
        &\qquad \qquad+[k-i][n-k-i]\Bigl([n][k-2]-[k]^2\Bigr)\biggr),
    \end{aligned}
\end{equation}
where $i=\partial(x,y)$.
\end{lemma}

\begin{proof}
Using (\ref{blin}), 
\begin{equation*}
    \bigl<B_{xy},B_{xy}\bigr>=q^{2i}[k-i][n-k-i]\bigl<B_{xy},\widehat{x}\bigr>-q^{2i}[n-k-i]\Bigl<B_{xy},\widehat{x\cap y}\Bigr>-q^{i}[k-i]\Bigl<B_{xy},\widehat{x+y}\Bigr>.
\end{equation*}
Evaluate the above equation using (\ref{binnerx}), (\ref{binnercap}), (\ref{binnerplus}) to obtain the result.
\end{proof}

\begin{lemma}
\label{innbc}
For $x,y\in X$ such that $1<\partial(x,y)<k$, we have 
\begin{equation*}
    \bigl<B_{xy},C_{xy}\bigr>=q^{2i+1}[k-i][n-k-i][i]^2\Bigl([n][k-2]-[k]^2\Bigr),
\end{equation*}
where $i=\partial(x,y)$.
\end{lemma}
\begin{proof}
Using (\ref{clin}),
\begin{equation*}
    \bigl<B_{xy},C_{xy}\bigr>=q[i-1]^2\bigl<B_{xy},\widehat{x}\bigr>+q^{i-1}\bigl<B_{xy},\widehat{y}\bigr>+q^{i}[i-1]\Bigl<B_{xy},\widehat{x\cap y}\Bigr>+[i-1]\Bigl<B_{xy},\widehat{x+y}\Bigr>.
\end{equation*}
Evaluate the above equation using (\ref{binnerx}), (\ref{binnery}), (\ref{binnercap}), (\ref{binnerplus}) to obtain the result.
\end{proof} 

\begin{lemma}
\label{inncc}
For $x,y\in X$ such that $1<\partial(x,y)<k$, we have 
\begin{equation*}
    \bigl<C_{xy},C_{xy}\bigr>=[i]^2\biggl(q^{k-2}[n]\Bigl(2q[i-1]+q+1\Bigr)+[i]^2\Bigl([n][k-2]-[k]^2\Bigr)\biggr),
\end{equation*} where $i=\partial(x,y)$.
\end{lemma}

\begin{proof}
Using (\ref{clin}),
\begin{equation*}
    \bigl<C_{xy},C_{xy}\bigr>=q[i-1]^2\bigl<C_{xy},\widehat{x}\bigr>+q^{i-1}\bigl<C_{xy},\widehat{y}\bigr>+q^{i}[i-1]\Bigl<C_{xy},\widehat{x\cap y}\Bigr>+[i-1]\Bigl<C_{xy},\widehat{x+y}\Bigr>.
\end{equation*}
Evaluate the above equation using (\ref{cinnerx}), (\ref{cinnery}), (\ref{cinnercap}), (\ref{cinnerplus}) to obtain the result.
\end{proof}

\begin{theorem}
\label{innprodmat22}
Let $x,y\in X$ satisfy $1<\partial(x,y)<k$. In the following table, for each vector $u$ in the header column, and each vector $v$ in the header row, the $(u,v)$-entry of the table gives the inner product $\left<u,v\right>$. Write $i=\partial(x,y)$.
\begin{center}
    \begin{tabular}{c|c c c c}
        $\left<\;,\;\right>$ & $\widehat{x}$ & $\widehat{y}$ & $B_{xy}$ & $C_{xy}$\\ 
         \hline
         \\
         $\widehat{x}$ & $q^k[k][n-k]$ & $[n][k-i]-[k]^2$ & $\substack{q^{2i+1}[k-i][n-k-i]\cdot\\ \left([n][k-1]-[k]^2\right)}$ & $\substack{[i]^2\left([n][k-1]-[k]^2\right)}$\\
         \\
        $\widehat{y}$ & $[n][k-i]-[k]^2$ & $q^{k}[k][n-k]$ & $\substack{q^{2i+1}[k-i][n-k-i]\cdot\\ \left([n][k-i-1]-[k]^2\right)}$ & $\substack{[i]^2\left([n][k-i+1]-[k]^2\right)}$\\
        \\
        $B_{xy}$&$\substack{q^{2i+1}[k-i][n-k-i]\cdot\\\left([n][k-1]-[k]^2\right)}$ & $\substack{q^{2i+1}[k-i][n-k-i]\cdot\\\left([n][k-i-1]-[k]^2\right)}$ & $\substack{q^{4i+2}[k-i][n-k-i]\cdot\\ \bigl(q^{k-i-2}[n]\left([k-i]+[n-k-i]\right)+\\
        [k-i][n-k-i]\left([n][k-2]-[k]^2\right)\bigr)}$ & $\substack{q^{2i+1}[k-i][n-k-i]\cdot\\ [i]^2\left([n][k-2]-[k]^2\right)}$\\
        \\
        $C_{xy}$&$\substack{[i]^2\left([n][k-1]-[k]^2\right)}$ & $\substack{[i]^2\left([n][k-i+1]-[k]^2\right)}$ & $\substack{q^{2i+1}[k-i][n-k-i]\cdot\\ [i]^2\left([n][k-2]-[k]^2\right)}$ & $\substack{[i]^2\bigl(q^{k-2}[n]\left(2q[i-1]+q+1\right)+\\
    [i]^2\left([n][k-2]-[k]^2\right)\bigr)}$\\
    \end{tabular}
\end{center}
\end{theorem}
\begin{proof}
Combine Corollary \ref{cor2} and Lemmas \ref{cor1}, \ref{bcx}, \ref{bcy}, \ref{innbb}--\ref{inncc}.
\end{proof}

\section{The combinatorial basis for $\Fix(x,y)$}
\label{more}

Pick $x,y\in X$ such that $1<\partial(x,y)<k$. In this section we prove that the vectors in (\ref{basis2}) form a basis for the subspace $\Fix(x,y)$. We formally define the combinatorial basis for $\Fix(x,y)$. We also present the transition matrices between the geometric basis and the combinatorial basis for $\Fix(x,y)$.

\begin{theorem}
\label{2basis}
For $x,y\in X$ such that $1<\partial(x,y)<k$, the vectors in {\rm{(\ref{basis2})}} form a basis for $\Fix(x,y)$.
\end{theorem}

\begin{proof}
    We first find the matrix of coefficients when we write the vectors in (\ref{basis2}) as linear combinations in the geometric basis for $\Fix(x,y)$. From Lemma \ref{bc}, we routinely obtain the following matrix of coefficients: 
    \begin{equation}
        \label{tran2}
        \begin{pmatrix}
        1 & 0 & q^{2i}[k-i][n-k-i] & q[i-1]^2\\
        0 & 1 & 0 & q^{i-1}\\
        0 & 0 & -q^{2i}[n-k-i] & q^i[i-1]\\
        0 & 0 &-q^i[k-i]&[i-1]
        \end{pmatrix},
    \end{equation}
    where $i=\partial(x,y)$.
    
    It suffices to show that the determinant of this matrix is nonzero.

    The determinant is equal to $-q^{k+i}[i-1][n-2k]$. We have $[i-1]\neq 0$ since $1<i<k$. Also, $[n-2k]\neq 0$ since $n>2k$. Hence the determinant of the matrix (\ref{tran2}) is nonzero. The result follows.
\end{proof}

\begin{definition}
    Let $x,y\in X$ satisfy $1<\partial(x,y)<k$. By the \emph{combinatorial basis for $\Fix(x,y)$}, we mean the basis formed by the vectors in (\ref{basis2}).
\end{definition}

Next we give the transition matrices between the geometric basis and the combinatorial basis for $\Fix(x,y)$. Throughout this paper we will use the convention described in \cite[p.~352]{Ter1} for transition matrices.

\begin{theorem}
\label{trans2}
    Let $x,y\in X$ satisfy $1<\partial(x,y)<k$. Write $i=\partial(x,y)$. The transition matrix from the geometric basis to the combinatorial basis for $\Fix(x,y)$ is equal to the matrix {\rm{(\ref{tran2})}}.

    The transition matrix from the combinatorial basis to the geometric basis for $\Fix(x,y)$ is equal to
    \begin{equation}
        \label{inversetran2}
        \begin{pmatrix}
        1 & 0 & \frac{[k-i][n-k-1]}{q^{k-1}[n-2k]} & \frac{-[k-1][n-k-i]}{q^{k-i-1}[n-2k]}\\
        0 & 1 & \frac{[k-i]}{q^{k-i+1}[i-1][n-2k]} & \frac{-[n-k-i]}{q^{k-2i+1}[i-1][n-2k]}\\
        0 & 0 & \frac{-1}{q^{k+i}[n-2k]} & \frac{1}{q^k[n-2k]}\\
        0 & 0 &\frac{-[k-i]}{q^k[i-1][n-2k]}&\frac{[n-k-i]}{q^{k-i}[i-1][n-2k]}\\
        \end{pmatrix}.
    \end{equation}
\end{theorem}

\begin{proof}
The first assertion is immediate from the construction of the matrix (\ref{tran2}). For the second assertion, take the inverse of the matrix (\ref{tran2}) to obtain the matrix (\ref{inversetran2}).
\end{proof}

\begin{theorem}
\label{maintheorem}
    For $x,y\in X$ such that $1<\partial(x,y)<k$, we have
    \begin{equation*}
        \begin{aligned}
            \widehat{x\cap y}&=\frac{[k-i][n-k-1]}{q^{k-1}[n-2k]}\widehat{x}+\frac{[k-i]}{q^{k-i+1}[i-1][n-2k]}\widehat{y}+\frac{-1}{q^{k+i}[n-2k]}B_{xy}+\frac{-[k-i]}{q^k[i-1][n-2k]}C_{xy},\\
            \widehat{x+y}&=\frac{-[k-1][n-k-i]}{q^{k-i-1}[n-2k]}\widehat{x}+\frac{-[n-k-i]}{q^{k-2i+1}[i-1][n-2k]}\widehat{y}+\frac{1}{q^k[n-2k]}B_{xy}+\frac{[n-k-i]}{q^{k-i}[i-1][n-2k]}C_{xy},
        \end{aligned}
    \end{equation*}
    where $i=\partial(x,y)$.
\end{theorem}

\begin{proof}
    Routine from the matrix (\ref{inversetran2}).
\end{proof}

\section{The subspace $\overline{\Fix}(x,y)$; swapping $x$ and $y$}
Pick $x,y\in X$ such that $1<\partial(x,y)<k$. In this section we consider an element $\sigma\in GL(V)$ that swaps $x$ and $y$. We show that the action of $\sigma$ on $\Fix(x,y)$ has eigenvalues $-1$ and $1$. We describe the corresponding eigenspaces. The eigenspace that corresponds to the eigenvalue $1$ will be denoted by $\overline{\Fix}(x,y)$. We construct a basis for $\overline{\Fix}(x,y)$, called the geometric basis; this basis is derived from the geometric basis for $\Fix(x,y)$. We construct another basis for $\overline{\Fix}(x,y)$, called the combinatorial basis; this basis is derived from the combinatorial basis for $\Fix(x,y)$. We give the transition matrices between the geometric basis and the combinatorial basis for $\overline{\Fix}(x,y)$.
\begin{lemma}
    \label{exist}
    For $x,y\in X$ such that $1<\partial(x,y)<k$, there exists $\sigma\in GL(V)$ that swaps $x$ and $y$.
\end{lemma}
\begin{proof}
    Immediate from the fact that the graph $\Gamma$ is distance-transitive.
\end{proof}

\begin{lemma}
\label{notelem}
    Let $x,y\in X$ satisfy $1<\partial(x,y)<k$. Pick $\sigma \in GL(V)$ that swaps $x$ and $y$. The following {\rm{(i)--(iii)}} hold:
    \begin{enumerate}[label={\rm{(\roman*)}}]
        \item $\Fix(x,y)$ is invariant under $\sigma$;

        \item on $\Fix(x,y)$, we have $\sigma^2={\rm{id}}$ and $\sigma\neq \pm {\rm{id}}$, where ${\rm{id}}$ is the identity element of $GL(V)$;

        \item the restriction of $\sigma$ to $\Fix(x,y)$ is diagonalizable with eigenvalues $-1$ and $1$.
    \end{enumerate}
\end{lemma}

\begin{proof}
(i) Immediate from the construction of $\sigma$.

(ii) The first assertion is immediate from the construction of $\sigma$. We now prove the second assertion. Note that $\widehat{x}-\widehat{y}$ is a nonzero vector in $\Fix(x,y)$ such that 
\begin{equation*}
    \sigma\bigl(\widehat{x}-\widehat{y}\bigr)=\widehat{y}-\widehat{x}.
\end{equation*}
Hence, $\sigma\neq {\rm{id}}$ on $\Fix(x,y)$. 
Note that $\widehat{x}+\widehat{y}$ is a nonzero vector in $\Fix(x,y)$ such that
\begin{equation*}
    \sigma\bigl(\widehat{x}+\widehat{y}\bigr)=\widehat{x}+\widehat{y}.
\end{equation*}
Hence, $\sigma\neq -{\rm{id}}$ on $\Fix(x,y)$. 
The result follows.

(iii) Immediate from (ii) of this lemma.
\end{proof}

We now find bases for the eigenspaces corresponding to the eigenvalues listed in Lemma \ref{notelem}(iii).

\begin{lemma}
\label{sigmalem}
Let $x,y\in X$ satisfy $1<\partial(x,y)<k$. Pick $\sigma\in GL(V)$ that swaps $x$ and $y$. For the action of $\sigma$ on $\Fix(x,y)$ the following hold:
\begin{enumerate}[label={\rm{(\roman*)}}]
    \item the eigenspace with eigenvalue $-1$ has a basis $\widehat{x}-\widehat{y}$;
    \item the following vectors form a basis for the eigenspace with eigenvalue $1$.
\end{enumerate} 
\begin{equation}
\label{barbasis1}
    \widehat{x}+\widehat{y},\qquad \qquad \widehat{x\cap y},\qquad \qquad \widehat{x+y}
\end{equation}
\end{lemma}
\begin{proof}
For the action of $\sigma$ on $\Fix(x,y)$, the eigenspace with eigenvalue $-1$ contains the vector $\widehat{x}-\widehat{y}$, and the eigenspace with eigenvalue $1$ contains the three vectors in (\ref{barbasis1}).

Recall the geometric basis for $\Fix(x,y)$ given in Theorem \ref{1basis}. Adjusting this basis, we find that the following vectors form a basis for $\Fix(x,y)$:
\begin{equation*}
    \widehat{x}-\widehat{y},\qquad \qquad \widehat{x}+\widehat{y}, \qquad \qquad \widehat{x\cap y}, \qquad \qquad \widehat{x+y}.
\end{equation*}

The result follows.
\end{proof}

For $x,y\in X$ such that $1<\partial(x,y)<k$, the two eigenspaces described in Lemma \ref{sigmalem} are independent of the choice of $\sigma\in GL(V)$ that swaps $x$ and $y$.

\begin{definition}
\label{defbarfix}
For $x,y\in X$ such that $1<\partial(x,y)<k$, define $\overline{\Fix}(x,y)$ to be the eigenspace for the eigenvalue $1$ described in Lemma \ref{sigmalem}(ii). 
\end{definition}
Note that $\overline{\Fix}(x,y)$ has a basis (\ref{barbasis1}). 

\begin{definition}
    Let $x,y\in X$ satisfy $1<\partial(x,y)<k$. By the \emph{geometric basis for $\overline{\Fix}(x,y)$}, we mean the basis (\ref{barbasis1}).
\end{definition}

Pick $x,y\in X$ such that $1<\partial(x,y)<k$. We just introduced the geometric basis for $\overline{\Fix}(x,y)$. We now use the combinatorial basis for $\Fix(x,y)$ to find another basis for $\overline{\Fix}(x,y)$. To find this basis, we recall the balanced set condition. 

\begin{lemma}{\rm{\cite[Theorem~3.3]{Ter2}}}
\label{balancebc}
For $x,y\in X$ such that $1<\partial(x,y)<k$, 
\begin{equation}
    \nonumber
    B_{xy}-B_{yx}=\zeta\bigl(\widehat{x}-\widehat{y}\bigr),\qquad \qquad C_{xy}-C_{yx}=\xi\bigl(\widehat{x}-\widehat{y}\bigr), 
\end{equation}
where 
\begin{equation}
\begin{aligned}
    \nonumber
    \zeta&=q^{2i}[k-i][n-k-i],\qquad \qquad \xi=q[i][i-2],\qquad \qquad i=\partial(x,y).
\end{aligned}
\end{equation}
\end{lemma}

\begin{definition}
\label{barbcdef}
For $x,y\in X$ such that $1<\partial(x,y)<k$, define 
\begin{equation}
    \nonumber
    \overline{B_{xy}}=B_{xy}-\zeta\widehat{x},\qquad \qquad  \overline{C_{xy}}=C_{xy}-\xi\widehat{x},
\end{equation}
where $\zeta$, $\xi$ are from Lemma \ref{balancebc}.
\end{definition}

By Lemma \ref{balancebc}, 
\begin{equation}
\label{barbca}
\overline{B_{xy}}=\overline{B_{yx}},\qquad \qquad \overline{C_{xy}}=\overline{C_{yx}}.
\end{equation}

\begin{lemma}
\label{barcontain}
For $x,y\in X$ such that $1<\partial(x,y)<k$, the subspace $\overline{\Fix}(x,y)$ contains the vectors $\overline{B_{xy}}$,  $\overline{C_{xy}}$.
\end{lemma}
\begin{proof}
By Theorem \ref{2basis} the subspace $\Fix(x,y)$ contains the vectors $\overline{B_{xy}}$ and $\overline{C_{xy}}$. Let $\sigma\in GL(V)$ swap $x$ and $y$. By construction, $\sigma\Bigl(\overline{B_{xy}}\Bigr)=\overline{B_{yx}}$ and $\sigma\Bigl(\overline{C_{xy}}\Bigr)=\overline{C_{yx}}$. By (\ref{barbca}), the vectors $\overline{B_{xy}}$ and $\overline{C_{xy}}$ are fixed by $\sigma$. The result follows.
\end{proof}

\begin{lemma}
\label{barbc}
For $x,y\in X$ such that $1<\partial(x,y)<k$, 
\begin{equation}
\begin{aligned}
    \nonumber
    \overline{B_{xy}}&=-q^{2i}[n-k-i]\widehat{x\cap y}-q^i[k-i]\widehat{x+y},\\
    \overline{C_{xy}}&=q^{i-1}(\widehat{x}+\widehat{y})+q^{i}[i-1]\widehat{x\cap y}+[i-1]\widehat{x+y},
\end{aligned}
\end{equation}
where $i=\partial(x,y)$.
\begin{proof}
Routine using Lemma \ref{bc} and Definition \ref{barbcdef}. 
\end{proof}
\end{lemma}

\begin{theorem}
\label{2barbasis}
    For $x,y\in X$ such that $1<\partial(x,y)<k$, the following vectors form a basis for $\overline{\Fix}(x,y)$:
    \begin{equation}
        \label{barbasis2}
        \widehat{x}+\widehat{y},\qquad \qquad \overline{B_{xy}},\qquad \qquad \overline{C_{xy}}.
    \end{equation}
\end{theorem}

\begin{proof}
   We first find the matrix of coefficients when we write the vectors in (\ref{barbasis2}) as linear combinations in the geometric basis for $\overline{\Fix}(x,y)$. From Lemma \ref{barbc}, we routinely obtain the following matrix of coefficients:
    \begin{equation}
    \label{bartran2}
        \begin{pmatrix}
        1 & 0 & q^{i-1}\\
        0 & -q^{2i}[n-k-i] & q^i[i-1]\\
        0 & -q^{i}[k-i] & [i-1]\\
        \end{pmatrix},
    \end{equation}
    where $i=\partial(x,y)$.
    
    It suffices to show that the determinant of this matrix is nonzero.
    
    The determinant is equal to $-q^{k+i}[i-1][n-2k]$. We have $[i-1]\neq 0$ since $1<i<k$. Also, $[n-2k]\neq 0$ since $n>2k$. Hence the determinant of the matrix (\ref{bartran2}) is nonzero. The result follows.
\end{proof}

\begin{definition}
    Let $x,y\in X$ satisfy $1<\partial(x,y)<k$. By the \emph{combinatorial basis for $\overline{\Fix}(x,y)$}, we mean the basis formed by the vectors in (\ref{barbasis2}).
\end{definition}

Next we give the transition matrices between the geometric basis and the combinatorial basis for $\overline{\Fix}(x,y)$. 

\begin{theorem}
\label{bartrans2}
Let $x,y\in X$ satisfy $1<\partial(x,y)<k$. Write $i=\partial(x,y)$. The transition matrix from the geometric basis to the combinatorial basis for $\overline{\Fix}(x,y)$ is equal to the matrix {\rm{(\ref{bartran2})}}.

The transition matrix from the combinatorial basis to the geometric basis for $\overline{\Fix}(x,y)$ is equal to
\begin{equation}
    \label{inversebartran2}
    \mathlarger{\begin{pmatrix}
    1 & \frac{[k-i]}{q^{k-i+1}[i-1][n-2k]} & \frac{-[n-k-i]}{q^{k-2i+1}[i-1][n-2k]} \\
    0 & \frac{-1}{q^{k+i}[n-2k]} & \frac{1}{q^{k}[n-2k]} \\
    0 & \frac{-[k-i]}{q^{k}[i-1][n-2k]} & \frac{[n-k-i]}{q^{k-i}[i-1][n-2k]}
    \end{pmatrix}}.
\end{equation}
\end{theorem}

\begin{proof}
The first assertion is immediate from the construction of the matrix (\ref{bartran2}). For the second assertion, take the inverse of the matrix (\ref{bartran2}) to obtain the matrix (\ref{inversebartran2}).
\end{proof}

\begin{theorem}
\label{barmaintheorem}
    For $x,y\in X$ such that $1<\partial(x,y)<k$, we have
    \begin{equation*}
        \begin{aligned}
            \widehat{x\cap y}&=\frac{[k-i]}{q^{k-i+1}[i-1][n-2k]}\bigl(\widehat{x}+\widehat{y}\bigr)+\frac{-1}{q^{k+i}[n-2k]}\overline{B_{xy}}+\frac{-[k-i]}{q^{k}[i-1][n-2k]}\overline{C_{xy}},\\
            \widehat{x+y}&=\frac{-[n-k-i]}{q^{k-2i+1}[i-1][n-2k]}\bigl(\widehat{x}+\widehat{y}\bigr)+\frac{1}{q^{k}[n-2k]}\overline{B_{xy}}+\frac{[n-k-i]}{q^{k-i}[i-1][n-2k]}\overline{C_{xy}},
        \end{aligned}
    \end{equation*}
    where $i=\partial(x,y)$.
\end{theorem}

\begin{proof}
    Routine from the matrix (\ref{inversebartran2}).
\end{proof}

\section{The subspace $\Fix(x\cap y,x+y)$}
Pick distinct $x,y\in X$ such that $1<\partial(x,y)<k$. In this section we describe the subspace $\Fix(x\cap y,x+y)$. We show that the subspace $\Fix(x\cap y,x+y)$ is contained in $\overline{\Fix}(x,y)$. We construct a basis for $\Fix(x\cap y,x+y)$, called the combinatorial basis; this basis is derived from the combinatorial basis for $\overline{\Fix}(x,y)$. We give the transition matrices between the geometric basis and the combinatorial basis for $\Fix(x\cap y,x+y)$. At the end of the section, we describe the orthogonal complement of $\Fix(x\cap y,x+y)$ in $\overline{\Fix}(x,y)$.

\begin{theorem}
\label{1checkbasis}
For $x,y\in X$ such that $1<\partial(x,y)<k$, the following vectors form a basis for $\Fix(x\cap y, x+y)$:
\begin{equation}
        \label{checkbasis1}
        \widehat{x\cap y},\qquad \qquad \widehat{x+y}. 
    \end{equation}
\end{theorem}

\begin{proof}
    Since $\partial(x,y)<k$, we have $x\cap y\neq 0$. Note that $x\cap y\subseteq x+y$. Since $n>2k$, we have $x+y\neq V$. The result follows from Case 5 of Lemma \ref{main2}.
\end{proof}

In view of Definition \ref{geometricdef}, the basis (\ref{checkbasis1}) is the geometric basis for $\Fix(x\cap y,x+y)$.

\begin{corollary}
    \label{13.1cor}
    For $x,y\in X$ such that $1<\partial(x,y)<k$, the subspace $\Fix(x\cap y,x+y)$ is contained in $\overline{\Fix}(x,y)$. 
\end{corollary}

\begin{proof}
    The geometric basis for $\Fix(x\cap y,x+y)$ is a subset of (\ref{barbasis1}). The vectors in (\ref{barbasis1}) form a basis for $\overline{\Fix}(x,y)$. The result follows.
\end{proof}

In Theorem \ref{1checkbasis}, we gave a basis for $\Fix(x\cap y,x+y)$. Our next goal is to use the combinatorial basis for $\overline{\Fix}(x,y)$ to find another basis for $\Fix(x\cap y, x+y)$.

\begin{definition}
\label{checkbcdef}
For $x,y\in X$ such that $1<\partial(x,y)<k$, define 
\begin{equation}
\nonumber
    \widecheck{B_{xy}}=\overline{B_{xy}},\qquad \qquad \widecheck{C_{xy}}=\overline{C_{xy}}-q^{i-1}\bigl(\widehat{x}+\widehat{y}\bigr),
\end{equation}
where $i=\partial(x,y)$.
\end{definition}

\begin{lemma}
\label{checkbc}
For $x,y\in X$ such that $1<\partial(x,y)<k$, 
\begin{equation}
\begin{aligned}
    \nonumber
    \widecheck{B_{xy}}&=-q^{2i}[n-k-i]\widehat{x\cap y}-q^i[k-i]\widehat{x+y},\\
    \widecheck{C_{xy}}&=q^{i}[i-1]\widehat{x\cap y}+[i-1]\widehat{x+y},
\end{aligned}
\end{equation}
where $i=\partial(x,y)$.

\begin{proof}
Routine using Lemma \ref{barbc} and Definition \ref{checkbcdef}.
\end{proof}
\end{lemma}

\begin{lemma}
    \label{checkcontain}
    For $x,y\in X$ such that $1<\partial(x,y)<k$, the subspace $\Fix(x\cap y,x+y)$ contains the following vectors:
\begin{equation}
\label{checkbasis2}
    \widecheck{B_{xy}},\qquad \widecheck{C_{xy}}.
\end{equation}
\end{lemma}

\begin{proof}
    Referring to Lemma \ref{checkbc}, the vectors $\widecheck{B_{xy}}$ and $\widecheck{C_{xy}}$ are linear combinations in (\ref{checkbasis1}). The result follows from Theorem \ref{1checkbasis}.
\end{proof}

\begin{theorem}
\label{2checkbasis}
For $x,y\in X$ such that $1<\partial(x,y)<k$, the two vectors in {\rm{(\ref{checkbasis2})}} form a basis for $\Fix(x\cap y, x+y)$.
\end{theorem}

\begin{proof}    
    We first find the matrix of coefficients when we write the vectors in (\ref{checkbasis2}) as linear combinations in the geometric basis for $\Fix(x\cap y,x+y)$. From Lemma \ref{checkbc}, we routinely obtain the following matrix of coefficients:
    \begin{equation}
    \label{checktran2}
    \begin{pmatrix}
    -q^{2i}[n-k-i] & q^i[i-1]\\
    -q^i[k-i] & [i-1]\\
    \end{pmatrix},
    \end{equation}
    where $i=\partial(x,y)$.
    
    It suffices to show that the determinant of this matrix is nonzero.
    
    The determinant is equal to $-q^{k+i}[i-1][n-2k]$. We have $[i-1]\neq 0$ since $1<i<k$. Also, $[n-2k]\neq 0$ since $n>2k$. Hence the determinant of the matrix (\ref{checktran2}) is nonzero. The result follows.
\end{proof}

\begin{definition}
    Let $x,y\in X$ satisfy $1<\partial(x,y)<k$. By the \emph{combinatorial basis for $\Fix(x\cap y,x+y)$}, we mean the basis formed by the vectors in $(\ref{checkbasis2})$.
\end{definition}

Next we display the transition matrices between the geometric basis and the combinatorial basis for $\Fix(x\cap y,x+y)$.

\begin{theorem}
\label{checktrans2}
Let $x,y\in X$ satisfy $1<\partial(x,y)<k$. Write $i=\partial(x,y)$. The transition matrix from the geometric basis to the combinatorial basis for $\Fix(x\cap y,x+y)$ is equal to the matrix {\rm{(\ref{checktran2})}}.

The transition matrix from the combinatorial basis to the geometric basis for $\Fix(x\cap y,x+y)$ is equal to
\begin{equation}
    \label{inversechecktran2}
    \mathlarger{\begin{pmatrix}
    \frac{-1}{q^{k+i}[n-2k]} & \frac{1}{q^k[n-2k]}\\
    \frac{-[k-i]}{q^k[i-1][n-2k]} & \frac{[n-k-i]}{q^{k-i}[i-1][n-2k]} 
    \end{pmatrix}}.
\end{equation}
\end{theorem}

\begin{proof}
The first assertion is immediate from the construction of the matrix (\ref{checktran2}). For the second assertion, take the inverse of the matrix (\ref{checktran2}) to obtain the matrix (\ref{inversechecktran2}).
\end{proof}

\begin{theorem}
\label{checkmaintheorem}
    For $x,y\in X$ such that $1<\partial(x,y)<k$, we have
    \begin{equation*}
        \begin{aligned}
            \widehat{x\cap y}&=\frac{-1}{q^{k+i}[n-2k]}\widecheck{B_{xy}}+\frac{-[k-i]}{q^k[i-1][n-2k]}\widecheck{C_{xy}},\\
            \widehat{x+y}&=\frac{1}{q^k[n-2k]}\widecheck{B_{xy}}+\frac{[n-k-i]}{q^{k-i}[i-1][n-2k]}\widecheck{C_{xy}},
        \end{aligned}
    \end{equation*}
    where $i=\partial(x,y)$.
\end{theorem}

\begin{proof}
    Routine from the matrix (\ref{inversechecktran2}).
\end{proof}

Recall from Corollary \ref{13.1cor} that the subspace $\Fix(x\cap y,x+y)$ is contained in $\overline{\Fix}(x,y)$. Let $\Fix(x\cap y,x+y)^{\perp}$ denote the orthogonal complement of $\Fix(x\cap y, x+y)$ in $\overline{\Fix}(x,y)$. Our next goal is to find a basis for $\Fix(x\cap y,x+y)^{\perp}$. We will express this basis in terms of the geometric basis for $\overline{\Fix}(x,y)$ and also the combinatorial basis for $\overline{\Fix}(x,y)$.

\begin{lemma}
\label{dim1}
    For $x,y\in X$ such that $1<\partial(x,y)<k$, the subspace $\Fix(x\cap y,x+y)^{\perp}$ has dimension $1$.
\end{lemma}

\begin{proof}
    By Lemma \ref{sigmalem}(ii), the subspace $\overline{\Fix}(x,y)$ has dimension $3$. By Theorem \ref{1checkbasis}, the subspace $\Fix(x\cap y,x+y)$ has dimension $2$. The result follows.
\end{proof}

\begin{lemma}
\label{orth}
For $x,y\in X$ such that $1<\partial(x,y)<k$, the subspace $\Fix(x\cap y,x+y)^{\perp}$ has a basis
\begin{equation}
\label{perp}
    \bigl(q^i+1\bigr)\bigl(\widehat{x}+\widehat{y}\bigr)-2q^i\widehat{x\cap y}-2\widehat{x+y},
\end{equation}
where $i=\partial(x,y)$.
\end{lemma}

\begin{proof}
Using the table in Theorem \ref{innprodmat21} one routinely verifies that the vector (\ref{perp}) is contained in $\Fix(x\cap y,x+y)^{\perp}$. The three vectors in (\ref{barbasis1}) are linearly independent, so the vector (\ref{perp}) is nonzero. The result follows from Lemma \ref{dim1}.
\end{proof}

\begin{lemma}
    \label{orthcom}
    Referring to Lemma \ref{orth}, the vector {\rm{(\ref{perp})}} is equal to
\begin{equation*}
        \frac{1}{[i-1]}\Bigl(\bigl(q^{i-1}+1\bigr)[i]\bigl(\widehat{x}+\widehat{y}\bigr)-2\overline{C_{xy}}\Bigr),
    \end{equation*}
where $i=\partial(x,y)$.
\end{lemma}

\begin{proof}
    Apply the transition matrix (\ref{inversebartran2}) to the vector (\ref{perp}).
\end{proof}

\section{More comments about the uniqueness problem for $\Gamma$}
In the introduction, we mentioned the uniqueness problem for the Grassmann graph. In this section, we comment on how
this problem relates to our main results.

Consider the Grassmann graph $\Gamma=J_q(n,k)$ with $n>2k\geq 6$. Let $\Gamma'$ denote a distance-regular graph that has the same intersection numbers as $\Gamma$. Note that $\Gamma'$ and $\Gamma$ have the same eigenvalues. Recall the vertex set $X$ for $\Gamma$, and let $X'$ denote the vertex set for $\Gamma'$. 

Let $E'$ denote a Euclidean space of dimension $[n]-1$. By \cite[Lecture~13]{TerCoursenote}, there exists a Euclidean
representation $(E',\rho)$ of $\Gamma'$ associated with the eigenvalue $\theta_1$. We normalize this representation such that 
\begin{equation*}
    \bigl\Vert \rho(x')\bigr\Vert^2 = q^k[k][n-k]
\end{equation*}
for all $x'\in X'$.

Let us attempt to recover the projective geometry $P$
from $\Gamma'$. To motivate things, we first consider $\Gamma$. For $x,y\in X$ such that $1<\partial(x,y)<k$, we defined the sets $\mathcal{B}_{xy},\mathcal{C}_{xy}$ in Definition \ref{calbcdef}. In Definition \ref{bcdef}, we used $\mathcal{B}_{xy}, \mathcal{C}_{xy}$ to define the vectors $B_{xy}, C_{xy}$ in $E$. In Theorem \ref{maintheorem}, we wrote each vector $\widehat{x\cap y}, \widehat{x+y}$ as a linear combination of
\begin{equation*}
    \widehat{x},\qquad \qquad \widehat{y},\qquad \qquad B_{xy},\qquad \qquad  C_{xy}.
\end{equation*}

We now talk about the graph $\Gamma'$; in the Euclidean space $E'$, let us attempt to mimic the vectors $\widehat{x\cap y}, \widehat{x+y}$. Let $x',y'\in X'$ satisfy $\partial\bigl(x',y'\bigr)=\partial(x,y)$. Define the sets $\mathcal{B}_{x'y'}, \mathcal{C}_{x'y'}$ similarly to Definition \ref{calbcdef}. Define the vectors $B_{x'y'}, C_{x'y'}$ similarly to Definition \ref{bcdef}. For notational convenience, write $i= \partial(x',y')$.

Motivated by the first equation in Theorem \ref{maintheorem},
we bring in the vector
\begin{equation}
\label{rhocap}
    \frac{[k-i][n-k-1]}{q^{k-1}[n-2k]}\rho\bigl(x'\bigr)+\frac{[k-i]}{q^{k-i+1}[i-1][n-2k]}\rho\bigl(y'\bigr)+\frac{-1}{q^{k+i}[n-2k]}B_{x'y'}+\frac{-[k-i]}{q^k[i-1][n-2k]}C_{x'y'}.
\end{equation}

This is the vector in $E'$ that will mimic $\widehat{x\cap y}$. For motivational purposes, we denote the vector (\ref{rhocap}) by $\rho\bigl(x'\cap y'\bigr)$.

Motivated by the second equation in Theorem \ref{maintheorem}, we bring in the vector 
\begin{equation}
\label{rhoplus}
    \frac{-[k-1][n-k-i]}{q^{k-i-1}[n-2k]}\rho\bigl(x'\bigr)+\frac{-[n-k-i]}{q^{k-2i+1}[i-1][n-2k]}\rho\bigl(y'\bigr)+\frac{1}{q^k[n-2k]}B_{x'y'}+\frac{[n-k-i]}{q^{k-i}[i-1][n-2k]}C_{x'y'}.
\end{equation}
This is the vector in $E'$ that will mimic $\widehat{x+y}$. For motivational purposes, we denote the vector (\ref{rhoplus}) by $\rho\bigl(x'+y'\bigr)$. 

\begin{problem}
\label{problem1}
    Try to show that for all $z'\in X'$, the inner product
    \begin{equation}
    \label{innxyz}
        \Bigl<\rho\bigl(x'\cap y'\bigr),\rho\bigl(z'\bigr)\Bigr> 
    \end{equation}
    is equal to one of the values 
    \begin{equation}
    \label{kil}
        [n][k-i-\ell]-[k-i][k]\qquad \qquad (0\leq \ell\leq k-i).
    \end{equation}
    By Lemma $\ref{introlem1}$, this will happen if $\Gamma'$ is isomorphic to $\Gamma$.
\end{problem}

\begin{problem}
    We define a binary relation on $X'$ called partner. For $z',w'\in X'$, we say that $z',w'$ are partners whenever
    \begin{equation*}
    \Bigl<\rho\bigl(x'\cap y'\bigr),\rho\bigl(z'\bigr)\Bigr>=\Bigl<\rho\bigl(x'\cap y'\bigr),\rho\bigl(w'\bigr)\Bigr>.
    \end{equation*}
    By construction, partner is an equivalence relation. Try to show that the partner equivalence classes form an equitable partition of $X'$. If $\Gamma'$ is isomorphic to $\Gamma$ then this will happen because the partner equivalence classes are the orbits of $\Stab\bigl(x'\cap y'\bigr)$.
\end{problem}

\begin{problem}
    Consider the set of vertices $z'\in X'$ such that the inner product (\ref{innxyz}) is equal to (\ref{kil}) with $\ell=0$. Try to show that the subgraph of $\Gamma'$ induced on this set, is geodesically closed and has diameter $i$. If $\Gamma'$ is isomorphic to $\Gamma$ then this will happen because the set will consist of the vertices in $X'$ that contain $x'\cap y'$.
\end{problem}

\section{Appendix}
In this section we derive some linear algebra facts that are used in the main body of the paper.

\begin{lemma}
\label{sum0basis}
Let $W$ denote a vector space of finite positive dimension $D$. Suppose the vectors $\{\mu_i\}_{i=0}^{D}$ span $W$ and sum to $0$. Then any $D$ vectors from $\{\mu_i\}_{i=0}^{D}$ form a basis for $W$.
\end{lemma}
\begin{proof}
It suffices to show that the vectors $\{\mu_i\}_{i=1}^{D}$ form a basis for $W$.

Since the vectors $\{\mu_i\}_{i=0}^{D}$ sum to $0$, $\mu_0$ is in the span of $\{\mu_i\}_{i=1}^{D}$. Hence, 
\begin{equation*}
    \text{Span}\{\mu_i\}_{i=1}^{D}=\text{Span}\{\mu_i\}_{i=0}^{D}=W.
\end{equation*}

The vector space $W$ has dimension $D$, so by linear algebra, the vectors $\{\mu_i\}_{i=1}^{D}$ form a basis for $W$. The result follows.
\end{proof}

\begin{lemma}
\label{sum0}
Let $W$ denote a vector space of finite positive dimension $D$. Suppose the vectors $\{\mu_i\}_{i=0}^{D}$ span $W$ and satisfy
\begin{equation}
\label{sum0v}
    \sum_{i=0}^{D}\mu_{i}=0.
\end{equation}
Suppose the vectors $\{\nu_i\}_{i=0}^{D}$ span $W$. Then the following are equivalent.
\begin{enumerate}[label={\rm{(\roman*)}}]
    \item There exists $\sigma\in GL(W)$ such that 
\begin{equation}
\label{vtow}
    \sigma(\mu_i)=\nu_i, \qquad \qquad 0\leq i\leq D.
\end{equation}

    \item The following sum holds:
\begin{equation}
\label{sum0w}
\sum_{i=0}^{D}\nu_i=0.
\end{equation}
\end{enumerate}

Moreover, if {\rm{(i)}} and {\rm{(ii)}} hold, then the map $\sigma$ is unique.
\end{lemma}

\begin{proof} 
\noindent ${\rm (i) }\Rightarrow {\rm (ii)}$ Apply $\sigma$ to each side of (\ref{sum0v}), and evaluate the result using (\ref{vtow}).

\noindent ${\rm (ii) }\Rightarrow {\rm (i)}$ By (\ref{sum0v}) and Lemma \ref{sum0basis}, the vectors $\{\mu_i\}_{i=1}^{D}$ form a basis for $W$. 

By (\ref{sum0w}) and Lemma \ref{sum0basis}, the vectors $\{\nu_i\}_{i=1}^{D}$ form a basis for $W$. By linear algebra, there exists a map $\sigma\in GL(W)$ such that
\begin{equation}
\label{0ton-1}
    \sigma(\mu_i)=\nu_i, \qquad \qquad 1\leq i\leq D.
\end{equation}

Using (\ref{sum0v}), (\ref{sum0w}), (\ref{0ton-1}), we obtain 
\begin{equation*}
    \sigma(\mu_0)=\nu_0.
\end{equation*}

By these comments, (i) holds. We have shown the equivalence of (i) and (ii). 

Assume that (i) and (ii) hold. We now show that the map
$\sigma$ is unique. Let the map $\sigma' \in GL(W)$ satisfy (\ref{vtow}).

Using (\ref{vtow}), we obtain 
\begin{equation*}
    \bigl(\sigma-\sigma'\bigr)(\mu_i)=0, \qquad \qquad 0\leq i\leq D.
\end{equation*}

The vectors $\{\mu_i\}_{i=0}^{D}$ span $W$, so $\sigma-\sigma'=0$. Therefore $\sigma=\sigma'$.
\end{proof}

\section*{Acknowledgement}
The author is currently a graduate student at the University of Wisconsin-Madison. He would like to thank his advisor, Paul Terwilliger, for many valuable ideas and suggestions for this paper. The author would also like to thank the anonymous referees for their helpful comments.

\section*{Declarations}
\subsection*{Conflict of interest}
The author has no relevant financial or non-financial interests to disclose.

Ian Seong\\
Department of Mathematics\\
University of Wisconsin \\
480 Lincoln Drive \\
Madison, WI 53706-1388 USA \\
email: iseong@wisc.edu\\
\end{document}